\documentclass[12pt]{article}
\usepackage{amsmath}
\usepackage{graphicx}
\usepackage{enumerate}
\usepackage{natbib}
\usepackage{url} 
\usepackage{booktabs}
\usepackage{amsthm} 
\usepackage{nicefrac}

\usepackage{etoolbox}

\newcommand{\blind}{1}

\usepackage{amssymb}
\usepackage{xfrac} 
\usepackage{physics}
\usepackage{subcaption}
\usepackage{hyperref}
\usepackage{color}

\newtheorem{lemma}{Lemma}[section]
\counterwithin{lemma}{section}

\newtheorem{proposition}{Proposition}%
\newtheorem{remark}{Remark}%
\newtheorem{definition}{Definition}
\newtheorem{theorem}{Theorem}

\DeclareMathOperator*{\esssup}{ess\,sup}

\newcommand{\almosteverywhere}{\mathrm{a.e.}\;}

\begin{document}

\setcounter{section}{0}
\renewcommand{\thesection}{\arabic{section}}

\def\spacingset#1{\renewcommand{\baselinestretch}%
{#1}\small\normalsize} \spacingset{1}


\date{}
\if1\blind
{
  \title{\bf Numerical approximation of fractional diffusion equations on metric graphs}
  \author{
    David Bolin\thanks{The authors are listed alphabetically.}, 
    Lenin Riera-Segura\footnotemark[1], 
    and Alexandre B. Simas\footnotemark[1]\\
    Statistics Program, Computer, Electrical and Mathematical\\
    Sciences and Engineering (CEMSE) Division,\\
    King Abdullah University of Science and Technology (KAUST),\\
    Thuwal, 23955-6900, Kingdom of Saudi Arabia
  }
  \maketitle
} \fi

\if0\blind
{
  \bigskip
  \bigskip
  \bigskip
  \begin{center}
    {\LARGE\bf Numerical approximation of fractional diffusion equations on metric graphs}
\end{center}
  \medskip
} \fi

\bigskip
\begin{abstract}
    We study fractional diffusion equations on compact metric graphs, where the nonlocal dynamics is governed by fractional powers of the shifted Kirchhoff–Laplacian. Building on recent advances in the analysis of fractional operators on metric graphs, we establish a rigorous mathematical framework and propose a fully discrete scheme based on backward Euler time-stepping and finite element discretization. To approximate the action of the fractional operator, we employ rational approximations, reducing the problem to a sequence of sparse elliptic solves for efficient implementation. We derive error estimates for the temporal, spatial, and rational discretizations, and confirm convergence through numerical experiments.
\end{abstract}

\noindent%

\spacingset{1} 

\section{Introduction}
\label{sec:introduction}
In recent years, there has been growing interest in dynamical processes on networks \citep{Cordoni2017Stochasticreactiondiffusion}, driven by applications ranging from disease spread \citep{Pastor2015Epidemicprocess} and information propagation in social networks \citep{Du2018PDE} to gas pipeline systems \citep{Herty2010Anewmodelforgasflow}, among others, where the underlying dynamics are often modeled by partial differential equations (PDEs) \citep{bottcher2024dynamicalprocessesmetricnetworks}. These include a broad range of physical models, such as Schrödinger, wave, and Klein--Gordon equations \citep{Dutykh2018Wave, Goloshchapova2021Anonlinear, Goodman2019NLSbifurcations, Goodman2025QGlAB}, as well as stochastic space--time models \citep{Cordoni2017Stochasticreactiondiffusion, Kovacs2021Stochastics}. A fundamental class of such dynamics is diffusion, which serves as a core model for the propagation of mass, energy, or information across networked systems \citep{Barrat2008Dynamicalprocess, Jivkov2014Anetworkmodel}. Parabolic PDEs, in particular, frequently arise in this setting as they effectively capture the temporal evolution of diffusive processes on networks. For instance, Weller \citep{kups73182} employed parabolic equations on metric graphs to simulate the distribution of tau proteins in the brain's structural network, offering insights into the progression of Alzheimer’s disease. See \citep[Chapter~11]{Danko2017Modelelements} and \citep{Pastor2015Epidemicprocess, Raj2012Anetworkdiffusionmodel} for additional examples of diffusion modeling on networks. 

However, in many real-world scenarios, classical diffusion fails to capture the observed dynamics, especially when long-range interactions are present \citep{Daoud2025classp, Daoud2024Aclass}. This has led to increasing interest in fractional diffusion, where the standard diffusion operator is replaced by its fractional counterpart to better reflect anomalous transport phenomena \citep{Vazquez2017Themathematical}. For further discussion on anomalous diffusion across different domains, see \citep{Henry2006Anomalous, Miller2006Onthecontrollability, Riascos2014Fractional, Somathilake2018Aspacefractional} and the references therein.  An important application of fractional diffusion arises in the context of metric graphs, where it has become a powerful modeling tool for complex systems in which both anomalous transport and the underlying network topology significantly influence the dynamics. Central to these formulations is the fractional Laplacian, which serves as the primary operator governing the nonlocal diffusion dynamics. To date, studies of fractional diffusion on metric graphs have primarily focused on time-fractional models \citep{Faheem2023Acollocation, Kumari2025Finite, Mehandiratta2021Optimalcontrol}. Existing works on space-fractional diffusion are comparatively few \citep{Leugering2023Optimalcontrol} and, together with the time-fractional studies, have been restricted to relatively simple graph topologies, such as star graphs. In contrast, the present work considers space-fractional diffusion on general compact metric graphs and, to the best of our knowledge, provides the first numerical analysis for this class of problems.

Several recent works have advanced the theoretical and numerical analysis of fractional diffusion problems, including studies on well-posedness, controllability, and discretization techniques \citep{Ainsworth2018Towards, Antil2016Aspacetime, Bonito2017Theapproximation, Bonito2018Numericalmethods, Daoud2025classp, Daoud2024Aclass, Miller2006Onthecontrollability, Nochetto2016APDEapproach, Riascos2014Fractional, Somathilake2018Aspacefractional, Vazquez2017Themathematical, kups73182}. Among these, Glusa and Otárola \citep{Glusa2021errorestimates} addressed a control-constrained linear-quadratic optimal control problem for the fractional heat equation on Euclidean domains, introducing a numerical method based on an implicit finite difference discretization in time with a piecewise linear finite element discretization in space. Inspired by their work, we consider the associated state equation but now on a metric graph, and refine the numerical approximation by incorporating a rational approximation of the fractional power on top of the standard time and space discretizations. Specifically, we analyze and numerically approximate solutions to fractional diffusion equations on metric graphs of the form
\begin{equation}
\label{eq:maineq}
\left\{
\begin{aligned}
    \partial_t u(s,t) + (\kappa^2 - \Delta_\Gamma)^{\sfrac{\alpha}{2}} u(s,t) &= f(s,t), && \quad (s,t) \in \Gamma \times (0, T), \\
    u(s,0) &= u_0(s), && \quad s \in \Gamma,
\end{aligned}
\right.
\end{equation}
with $u(\cdot,t)$ satisfying the Kirchhoff vertex conditions
\begin{equation}
\label{eq:Kcond}
   \mathcal{K} =  \left\{\phi\in C(\Gamma)\cap \textstyle\bigoplus_{e\in\mathcal{E}} C^{1}(e)\;\middle|\; \forall v\in \mathcal{V}:\; \textstyle\sum_{s\in v}\partial \phi(s)=0 \right\}.
\end{equation}
Here $\Gamma = (\mathcal{V},\mathcal{E})$ is a compact metric graph, $\kappa>0$, $\alpha\in(0,2]$ determines the smoothness of $u(\cdot,t)$, $\Delta_{\Gamma}$ is the so-called Kirchhoff--Laplacian. Details will be provided in later sections. Further, $f:\Gamma\times (0,T)\to\mathbb{R}$ on the right-hand side and $u_0: \Gamma \to \mathbb{R}$ determining the initial condition are fixed functions.  Recent work by Bolin et al. \citep{bolin2023regularity} laid the foundation for the study of fractional elliptic equations on metric graphs, establishing well-posedness and regularity under Kirchhoff vertex conditions. By exploiting the spectral structure of the Kirchhoff--Laplacian and introducing fractional Sobolev spaces tailored to graph domains, they developed a rigorous functional framework. These contributions provide essential mathematical tools and pave the way for the analysis of fractional parabolic equations. In this work, we build upon these results to model fractional diffusion on metric graphs.

Turning back to the nonlocal nature of the fractional Laplacian, which encapsulates long-range interactions across the domain, we note that it gives rise to  unique analytical and computational challenges \citep{Antil2016Aspacetime}. On bounded domains in Euclidean space, several definitions of the fractional Laplacian coexist, each reflecting the operator’s nonlocal behavior in different ways and leading to distinct theoretical frameworks and specialized numerical methods. The integral definition, based on singular kernels and extensions by zero outside the domain, naturally leads to dense system matrices, which require specialized techniques such as matrix compression or hierarchical solvers to achieve computational efficiency~\citep{Ainsworth2018Towards, Bonito2018Numericalmethods, Glusa2021errorestimates}. An alternative approach is based on the spectral definition, where the fractional Laplacian is defined via eigenfunction expansions of the Dirichlet Laplacian. This definition enables the use of the Caffarelli–Silvestre extension~\citep{Caffarelli2007Extension} adapted to bounded domains, which reformulates the fractional Laplacian as a local, degenerate elliptic operator in a higher-dimensional space. The resulting extension problem is typically discretized using finite element methods on anisotropically graded meshes to resolve singularities near the boundary~\citep{Bonito2018Numericalmethods, Nochetto2016APDEapproach}. In the context of compact metric graphs, however, the geometry and topology differ from those of Euclidean domains, and the analytical and numerical techniques developed for these Euclidean-based formulations do not transfer directly to the graph setting. Building on the functional framework introduced by Bolin et al.~\citep{bolin2023regularity}, we instead adopt a spectral definition based on the eigenpairs of the shifted Kirchhoff--Laplacian $L=\kappa^2-\Delta_\Gamma$, whose domain naturally encodes the graph topology and the Kirchhoff vertex conditions $\mathcal{K}$. Fractional powers of the operator are approximated using rational functions, reducing the problem to a sequence of standard elliptic solves. All solves are performed using sparse matrices, preserving sparsity throughout the computations. This yields a sparse, efficient, and scalable method that avoids the complexities of domain extensions, nonlocal kernel evaluations, or matrix compression techniques.

The remainder of the paper is organized as follows. In Section~\ref{notandpre}, we introduce the notation and the theoretical framework, including the structure of compact metric graphs, function spaces, spectral properties of the shifted Kirchhoff--Laplacian. 
Section~\ref{sec:wellposedness} addresses the well-posedness of the problem~\eqref{eq:maineq}, proving existence, uniqueness, and regularity of solutions. In Section~\ref{numericalapproxsec}, we introduce a fully discrete numerical scheme combining backward Euler time-stepping with a finite element spatial discretization. 
These results are used to derive error estimates in time and space in Section~\ref{timeandspaceerroranalysis}. Section~\ref{rationalapproximationsection} introduces a rational approximation strategy for efficiently approximating the action of the fractional operator $L^{-\sfrac{\alpha}{2}}$ and provides error estimates that quantify the impact of this approximation on the overall numerical solution. Section~\ref{numericalimplementationsec} presents the numerical implementation and convergence results, supported by illustrative examples. The computations were performed using the \texttt{MetricGraph}~\citep{MetricGraphRpackage} and \texttt{rSPDE}~\citep{rSPDERpackage} packages within the \texttt{R}~\citep{Rsoftware} software environment. Finally, Section~\ref{conclusion} offers concluding remarks and discusses potential directions for future research. Code for reproducing all results is available at \texttt{https://github.com/leninrafaelrierasegura/NAFDEMG}.

\section{Notation and Preliminaries}
\label{notandpre}
%
Throughout the article, $\Gamma$ denotes a compact metric graph, defined as a combinatorial graph $\Gamma = (\mathcal{V}, \mathcal{E})$ equipped with a metric $d(\cdot,\cdot)$, where the set of vertices $\mathcal{V}=\{v\}$ is finite and the set of undirected edges $\mathcal{E} = \{e\}$ are rectifiable curves with positive finite length $\ell_e$. 
We assume that $\Gamma$ is connected and that the metric $d$ corresponds to the geodesic (shortest path) distance. For each vertex $v\in\mathcal{V}$, let $\mathcal{E}_{v}$ denote the collection of edges incident to $v$, and let $L_v=\{e\in\mathcal{E}_v:\; e \text{ is a loop}\}$. The degree of $v$ is defined by $\deg(v)=|\mathcal{E}_v|+|L_v|$. A location $s\in\Gamma$ is a point on some edge $e\in\mathcal{E}$ and may be expressed by the ordered pair $(e,t)$, $t\in[0,\ell_e]$, called a coordinate representation of $s$. A vertex $v\in\mathcal{V}$ admits $\deg(v)$ distinct coordinate representations, whereas any point $s\in\Gamma\setminus\mathcal{V}$ has a unique coordinate representation. We write $s\in v$ to indicate $s$ is a coordinate representation of the vertex $v$. A real-valued function $f$ on $\Gamma$ is specified as a family $\{f_e\}_{e\in\mathcal{E}}$ with edgewise-defined components $f_e:[0,\ell_e]\to\mathbb{R}$ satisfying  $f_e = f|_e$. 

We denote the space of continuous functions on $\Gamma$ by $C(\Gamma)$ and equip it with the supremum norm $\|f\|_{C(\Gamma)} = \sup_{s \in \Gamma} |f(s)| $. We denote the space of essentially bounded functions by $L^\infty(\Gamma)$ and equip it with the norm $\|f\|_{L^\infty(\Gamma)} = \esssup_{s \in \Gamma} |f(s)|$ defined as $\esssup_{s \in \Gamma} |f(s)| := \inf \{ M \geq 0 : \lambda(\{s \in \Gamma : |f(s)| > M\}) = 0 \}$. Here $ \lambda $ is the Lebesgue measure on $ \Gamma $, given by $ \lambda(A) = \sum_{e \in \mathcal{E}} \lambda_e(A \cap e) $, with $ \lambda_e $ being the Lebesgue measure on each edge $ e $, identified with a compact interval. From now on, we write $X(e)$ to denote a standard function space $X$ (e.g., $L_2$, $H^k$) on the interval $[0,\ell_e]$. For example, $C^1(e)$ in \eqref{eq:Kcond} denotes the space of continuously differentiable functions on $[0,\ell_e]$. The space of square-integrable functions on $\Gamma$, denoted by $L_2(\Gamma)$, consists of functions whose edgewise components are square-integrable functions. It is equipped with the inner product  $(f,g)_{L_2(\Gamma)} = \int_\Gamma f(s)g(s)\dd s = \sum_{e\in\mathcal{E}}\int_{e}f_{e}(s)g_{e}(s)\dd s$ and the corresponding norm $\norm{f}^2_{L_2(\Gamma)} = \sum_{e\in\mathcal{E}}\norm{f_{e}}^2_{L_2(e)}$. The Sobolev space $ H^1(\Gamma) $ is defined as $ \{f \in C(\Gamma) : \|f\|^2_{H^1(\Gamma)} < \infty \} = C(\Gamma) \cap \bigoplus_{e \in \mathcal{E}} H^1(e)$ and is equipped with the norm $\|f\|^2_{H^1(\Gamma)} = \sum_{e \in \mathcal{E}} (\|f'_{e}\|^2_{L_2(e)}+\|f_{e}\|^2_{L_2(e)})$, where $f_{e}'$ denotes the weak derivative of $ f_{e} $, characterized as the unique function in $L_2(e)$ such that  $f_e(x) = f_e(0) + \int_0^x f_e'(s)\dd s$ for all $x \in [0, \ell_{e}]$. This characterization implies that any edgewise component $ f_e \in H^1(e) $  admits a continuous representative. For $k\in\mathbb{N}$, we define the decoupled Sobolev space of order $k$ as  $\widetilde{H}^k(\Gamma) = \bigoplus_{e \in \mathcal{E}} H^k(e)$, with norm $\|f\|^2_{\widetilde{H}^k(\Gamma)} = \sum_{e \in \mathcal{E}} \|f_{e}\|^2_{H^k(e)}$. In particular, if $ f \in \widetilde{H}^2(\Gamma) $, then each $ f_e' $ can be identified with a continuous function.  However, functions in $\widetilde{H}^k(\Gamma)$ are not necessarily globally continuous across $\Gamma$. To ensure global continuity, we define the globally continuous Sobolev space $\widetilde{H}_C^k(\Gamma) = \widetilde{H}^k(\Gamma) \cap C(\Gamma)$. In particular, $ H^1(\Gamma) = \widetilde{H}^1(\Gamma) \cap C(\Gamma) = \widetilde{H}_C^1(\Gamma) $. To encode both second-order edgewise regularity and physically meaningful vertex conditions (namely, continuity and flux conservation), we introduce the space  $K(\Gamma) = \{f \in \widetilde{H}^2_C(\Gamma) : f\in \mathcal{K}\}$, where $\mathcal{K}$, as defined in \eqref{eq:Kcond}, denotes the set of functions satisfying the Kirchhoff conditions at each vertex. For a function $f\in\widetilde{H}^2(\Gamma)$, the directional derivative $ \partial f(s)$ at $s\in v$ along an edge $ e =[0,\ell_e]$ is defined as $ \partial f(s) = f_e'(0) $ if $ s=(e,0) $ and $ \partial f(s) = -f_e'(\ell_e) $ if $ s=(e, \ell_e) $. 

We now briefly recall the relevant interpolation theory. See \citep{ChandlerWilde2015interpolation} for more details. If $ E \subset F $ are two Hilbert spaces, we denote the real interpolation space of order $ s \in (0, 1) $ between $ F $ and $ E $ by $ (F, E)_s $. Using this, we define the fractional Sobolev spaces as $H^s(\Gamma) := (L_2(\Gamma), H^1(\Gamma))_s$ for $0 < s < 1$ and $H^s(\Gamma) := (H^1(\Gamma), \widetilde{H}_C^2(\Gamma))_{s-1}$ for $1 < s < 2$. Given two Hilbert spaces $(E, \norm{\cdot}_E)$ and $(F, \norm{\cdot}_F)$, we write $E\hookrightarrow F$ to indicate that $E$ is continuously embedded in $F$, i.e., $E\subset F$ and $\norm{\cdot}_{F}\lesssim \norm{\cdot}_{E}$. Here and in the sequel, the notation $A\lesssim B$ means $A\leq CB$ for some constant $C>0$, which may vary from line to line and is not of essential importance (e.g., it does not depend on discretization parameters). Given a function $\phi:\Gamma\times(0,T)\to\mathbb{R}$, the notation $\phi(t)$ refers to the function $\phi:(0,T)\ni t\mapsto \phi(t) := \phi(\cdot, t)\in X$, where $X$ is some Banach space. Additionally, for notational convenience, expressions such as $(\phi,\phi)_{L_2(\Gamma)}$ and $(\phi(t),\phi(t))_{L_2(\Gamma)}$ are shorthand for $(\phi(\cdot,t),\phi(\cdot,t))_{L_2(\Gamma)}$ for $\almosteverywhere t\in(0,T)$. We adopt the convention that matrices are denoted by bold capital letters (e.g., $\mathbf{A}$), their columns by single subscripts (e.g., $\mathbf{A}_k$), and their entries by double subscripts (e.g., $\mathbf{A}_{i,j}$).

The Kirchhoff--Laplacian $\Delta_{\Gamma}: D(\Delta_{\Gamma}) = K(\Gamma)\subset L_2(\Gamma)\to L_2(\Gamma)$ introduced in \eqref{eq:maineq} is defined as a second-order differential operator acting edgewise on functions in $K(\Gamma)$. That is, $ (\Delta_{\Gamma} f)|_e := f_e''$ for $f\in D(\Delta_{\Gamma})$. This domain ensures sufficient smoothness on each edge, continuity across the graph, and compatibility with Kirchhoff vertex conditions, making $\Delta_{\Gamma}$ a self-adjoint operator. For $\kappa>0$, we define the operator $L:D(L) = D(\Delta_{\Gamma})\to L_2(\Gamma)$ as the shifted Kirchhoff--Laplacian $L =\kappa^2 -\Delta_{\Gamma}$, which is densely-defined in $L_2(\Gamma)$, self-adjoint, and positive definite with a compact inverse, ensuring a well-posed spectral decomposition \citep{berkolaiko2013introduction, bolin2024gaussian}. Therefore, there exists a complete orthonormal system  $\{e_j\}_{j=1}^\infty$ on $L_2(\Gamma)$, with corresponding eigenvalues $\{\lambda_j\}_{j=1}^\infty$, that diagonalizes operator $L$. Moreover, for $\alpha>0$, the fractional operator $L^{\sfrac{\alpha}{2}}$  and inverse fractional operator $L^{-\sfrac{\alpha}{2}}$ are well-defined in the spectral sense. Specifically, the fractional operator $L^{\sfrac{\alpha}{2}} : D(L^{\nicefrac{\alpha}{2}}) \longmapsto L_2(\Gamma)$ is defined via $\phi \longmapsto L^{\nicefrac{\alpha}{2}}\phi = \sum_{j\in\mathbb{N}}\lambda_j^{\nicefrac{\alpha}{2}}(\phi, e_j)_{L_2(\Gamma)}e_j$, with domain $ D(L^{\nicefrac{\alpha}{2}}) = \dot{H}^{\alpha}(\Gamma) = \{\phi\in L_2(\Gamma): \norm{\phi}_{\dot{H}^{\alpha}}<\infty\}$, which is a Hilbert space  with inner product $(\phi,\psi)_\alpha = (L^{\nicefrac{\alpha}{2}} \phi, L^{\nicefrac{\alpha}{2}} \psi)_{L_2(\Gamma)}$ and induced norm $\|\phi\|_{\dot{H}^{\alpha}(\Gamma)}^2 = \|L^{\nicefrac{\alpha}{2}} \phi\|^2_{L_2(\Gamma)} = \sum_{j\in\mathbb{N}}\lambda_j^{\alpha}(\phi, e_j)^2_{L_2(\Gamma)}$. Using the fact that finite linear combinations of the eigenfunctions $\{e_j\}_{j=1}^\infty$ are dense in $L_2(\Gamma)$ and the Borel functional calculus from the spectral theorem applies, one can show that $L^{\sfrac{\alpha}{2}}$ inherits all the key properties of $L$. In particular, $L^{\sfrac{\alpha}{2}}$ is self-adjoint, densely-defined and bounded below. This implies that it is sectorial, and according to \citep[Theorem~1.3.4.]{Henry1981} $-L^{\sfrac{\alpha}{2}}$ is the infinitesimal generator of a compact analytic semigroup $T(t) :=\{\mathrm{e}^{-L^{\sfrac{\alpha}{2}}t}\}_{t\geq 0}$ of uniformly bounded linear operators on $L_2(\Gamma)$. Using the spectral decomposition, the semigroup has representation
\begin{align}
\label{spectral_semigroup}
   \textstyle T(t)\phi = \sum_{j\in\mathbb{N}}\mathrm{e}^{-\lambda^{\sfrac{\alpha}{2}}_jt}\left(\phi, e_j\right)_{L_2(\Gamma)}e_j,
\end{align}
and satisfies the contraction property $\norm{T(t)\phi}_{L_2(\Gamma)}\leq \norm{\phi}_{L_2(\Gamma)}$ for $t\geq0$.

We conclude the preliminaries by introducing the dual space of $\dot{H}^{\alpha}(\Gamma)$ as $\dot{H}^{-\alpha}(\Gamma)$ and endow it with the norm $\norm{\phi}_{\dot{H}^{-\alpha}(\Gamma)}^2  = \sum_{j\in\mathbb{N}}\lambda_j^{-\alpha}\langle\phi, e_j\rangle^2$, where $\langle\cdot,\cdot\rangle$ denotes the duality pairing between $\dot{H}^{-\alpha}(\Gamma)$ by $\dot{H}^{\alpha}(\Gamma)$. For $0\leq \alpha\leq \beta$, we have the continuous embeddings $\dot{H}^{\beta}(\Gamma)\hookrightarrow\dot{H}^{\alpha}(\Gamma)\hookrightarrow L_2(\Gamma) \hookrightarrow \dot{H}^{-\alpha}(\Gamma)\hookrightarrow\dot{H}^{-\beta}(\Gamma)$. Moreover, according to \citep[Lemma~2.1]{bolin2020numerical}, operator $L^{\nicefrac{\alpha}{2}}$ has a unique bounded extension to an isometric isomorphism $L^{\nicefrac{\alpha}{2}}: \dot{H}^{r}(\Gamma)\to \dot{H}^{r-\alpha}(\Gamma)$ for any $r\in\mathbb{R}$. This extension will play a crucial role in establishing the regularity properties of solutions to elliptic problems.
\section{Well-posedness}
\label{sec:wellposedness}
This section presents the regularity properties of the associated elliptic problem and the well-posedness of the fractional diffusion problem.
\subsection{Elliptic regularity}
Let $\alpha>0$ and consider the elliptic problem
\begin{align}
    \label{eq:elliptic_problem}
    L^{\sfrac{\alpha}{2}} u(s)=f(s),\quad s\in\Gamma,
\end{align}
where $u$ satisfies the Kirchhoff vertex conditions \eqref{eq:Kcond}. Let $f\in\dot{H}^{-\sfrac{\alpha}{2}}(\Gamma)$. A weak formulation for \eqref{eq:elliptic_problem} reads as follows: Find $u\in\dot{H}^{\sfrac{\alpha}{2}}(\Gamma)$ such that
\begin{align}
\label{weak_ellliptic_pro}
    \mathfrak{a}_{\alpha}(u, v) = \langle f,v\rangle,\quad \forall v\in \dot{H}^{\sfrac{\alpha}{2}}(\Gamma),
\end{align}
where the bilinear form $\mathfrak{a}_{\alpha}:\dot{H}^{\sfrac{\alpha}{2}}(\Gamma)\times \dot{H}^{\sfrac{\alpha}{2}}(\Gamma)\to \mathbb{R}$ is defined by $\mathfrak{a}_{\alpha}(\phi,\psi) := (L^{\sfrac{\alpha}{4}}\phi,L^{\sfrac{\alpha}{4}}\psi)_{L_2(\Gamma)}$ for $\phi,\psi\in\dot{H}^{\sfrac{\alpha}{2}}(\Gamma)$. Because $|\mathfrak{a}_{\alpha}(\phi,\psi)|\leq \|\phi\|_{\dot{H}^{\sfrac{\alpha}{2}}(\Gamma)}\|\psi\|_{\dot{H}^{\sfrac{\alpha}{2}}(\Gamma)}$ and $\mathfrak{a}_{\alpha}(\phi,\phi) = \|\phi\|^2_{\dot{H}^{\sfrac{\alpha}{2}}(\Gamma)}$, the bilinear form $\mathfrak{a}_{\alpha}(\cdot,\cdot)$ is continuous and coercive on $\dot{H}^{\sfrac{\alpha}{2}}(\Gamma)$. Therefore, the Lax--Milgram theorem ensures the well-posedness of problem \eqref{weak_ellliptic_pro}, guaranteeing the existence and uniqueness of the solution $u\in \dot{H}^{\sfrac{\alpha}{2}}(\Gamma)$ for any $f \in \dot{H}^{-{\sfrac{\alpha}{2}}}(\Gamma)$. However, under stronger regularity assumptions on the right-hand side, we obtain improved regularity of the solution, as stated below.
\begin{proposition}
\label{proposition:elliptic_regularity}
    Let $\alpha>0$ and $r\in\mathbb{R}$. If $f\in\dot{H}^{r}(\Gamma)$, then the solution $u$ to problem \eqref{eq:elliptic_problem} belongs to $\dot{H}^{\alpha+r}(\Gamma)$. Moreover, we have $\norm{u}_{\dot{H}^{\alpha+r}(\Gamma)} = \norm{f}_{\dot{H}^{r}(\Gamma)}$.
\end{proposition}
\begin{proof}
    Since $L^{\nicefrac{\alpha}{2}}$ is an isometric isomorphism from $\dot{H}^{\alpha+r}(\Gamma)$ to $ \dot{H}^{r}(\Gamma)$ for every $r\in\mathbb{R}$, we have that for every $f\in\dot{H}^{r}(\Gamma)$, 
    $u = L^{-\nicefrac{\alpha}{2}} f$ belongs to $\dot{H}^{\alpha+r}(\Gamma)$. In particular, $L^{\nicefrac{\alpha}{2}} u$ is well-defined as an element in $\dot{H}^r(\Gamma)$, and we have
    \begin{align*}
        \norm{u}_{\dot{H}^{\alpha+r}(\Gamma)} = \|L^{\nicefrac{\alpha}{2}}u\|_{\dot{H}^{r}(\Gamma)}= \norm{f}_{\dot{H}^{r}(\Gamma)}.
    \end{align*}
\end{proof}
Let $\mathbb{U} = \{\phi\in L_2(0,T;\dot{H}^{\sfrac{\alpha}{2}}(\Gamma)):\partial_t\phi\in L_2(0,T;\dot{H}^{-\sfrac{\alpha}{2}}(\Gamma))\}$. We define a weak solution of \eqref{eq:maineq} as follows.
\begin{definition}
    A function $u\in\mathbb{U}$ is called a \emph{weak solution} of \eqref{eq:maineq} if 
    \begin{equation}
\label{weakformofentireequation}
    \left\{
\begin{aligned}
    \langle\partial_tu,\phi\rangle + \mathfrak{a}_{\alpha}(u,\phi) &= \langle f,\phi\rangle,\quad\forall \phi\in \dot{H}^{\sfrac{\alpha}{2}}(\Gamma),\quad \almosteverywhere t\in(0,T),\\  
    u(0)&=u_0.
\end{aligned}
\right.
\end{equation}
\end{definition}
By the Lions--Magenes theorem applied to the Gelfand triple $\dot{H}^{\sfrac{\alpha}{2}}(\Gamma)\hookrightarrow L_2(\Gamma) \hookrightarrow \dot{H}^{-\sfrac{\alpha}{2}}(\Gamma)$, every $u\in\mathbb{U}$ admits a representative in $C([0,T];L_2(\Gamma))$ \citep{Evans2010Partial, Lions1972Nonhomogeneous}, so the initial condition $u(0)=u_0$ is well-defined.
\subsection{Existence and Uniqueness}
\label{existencesection}
To establish well-posedness, we rewrite \eqref{eq:maineq} as an abstract Cauchy problem in $L_2(\Gamma)$, given by
\begin{equation}
\label{eq:cauchyproblem}
    \left\{
\begin{aligned}
    \partial_tu(t)+L^{\sfrac{\alpha}{2}}u(t)&=f(t),\quad t\in (0, T),\\  
    u(0)&=u_0\in D(L^{\sfrac{\alpha}{2}}).
\end{aligned}
\right.
\end{equation}
The semigroup generated by $-L^{\sfrac{\alpha}{2}}$ provides the natural framework for analyzing \eqref{eq:maineq}. The next theorem gives the corresponding well-posedness result and identifies the mild solution with the weak solution introduced above.
\begin{theorem}
\label{theorem:existenceuniquenessandregularity}
    Let $\alpha>0$. If $u_0\in\dot{H}^{\alpha}(\Gamma)$ and $f\in L_2(0,T;L_2(\Gamma))$, then \eqref{eq:cauchyproblem} has a unique mild solution given by 
\begin{equation}
\label{eq:mildsolution}
        u(t) = T(t)u_0 + \int_0^t T(t-r)f(r)\dd r,\quad t\geq 0.
\end{equation}
Moreover, $u\in L_2(0,T;\dot{H}^{\alpha}(\Gamma))\cap H^1(0,T;L_2(\Gamma))$ and
\begin{align}
\label{ineq:boundonu}
\norm{u}_{L_2(0,T;\dot{H}^{\alpha}(\Gamma))}^2 +\norm{\partial_tu}_{L_2(0,T;L_2(\Gamma))}^2\lesssim \norm{u_0}_{\dot{H}^{\alpha}(\Gamma)}^2 + \norm{f}_{L_2(0,T;L_2(\Gamma))}^2.
\end{align}
Finally, the mild solution is the unique weak solution of
\eqref{weakformofentireequation}.
\end{theorem}
\begin{proof}
Existence and uniqueness of the mild solution follows as $-L^{\sfrac{\alpha}{2}}$ generates an analytic contraction semigroup on $L_2(\Gamma)$. The representation formula \eqref{eq:mildsolution} is given by the variation of parameters formula. The regularity properties and estimate \eqref{ineq:boundonu} are standard consequences of analytic semigroup theory. The stated regularity implies that $u$ satisfies \eqref{weakformofentireequation}, and hence the mild solution coincides with the weak.
\end{proof}
Using the representation \eqref{spectral_semigroup} of the semigroup $T(t)$, the solution \eqref{eq:mildsolution} is
\begin{equation}
\label{eq:sol_reprentation}
        u(s,t) = \displaystyle\sum_{j\in\mathbb{N}}\mathrm{e}^{-\lambda^{\sfrac{\alpha}{2}}_jt}\left(u_0, e_j\right)_{L_2(\Gamma)}e_j(s) + \int_0^t \displaystyle\sum_{j\in\mathbb{N}}\mathrm{e}^{-\lambda^{\sfrac{\alpha}{2}}_j(t-r)}\left(f(r), e_j\right)_{L_2(\Gamma)}e_j(s)\dd r.
\end{equation}
\section{Numerical approximation}
\label{numericalapproxsec}
We approximate \eqref{eq:maineq} using the implicit backward Euler method in time and piecewise linear finite elements on the graph in space.
\subsection{Time discretization}
    To discretize the time domain, we divide the interval $[0, T]$ into $N$ uniform subintervals with time steps $t_k = k\tau,\; k = 0, \dots, N$, where $\tau = T/N$ is the time step size. For a continuous function $\phi\in C([0,T];X)$, we denote its $X$-value at $t_k$ by $\phi^k=\phi(t_k)$ and let $\phi^\tau$ denote the sequence of evaluations $\{\phi^k\}_{k=0}^N$ at the time step sequence $\{t_k\}_{k=0}^N$. For any sequence $\phi^\tau$, we define the norms $\norm{\phi^\tau}_{\ell^\infty(X)} = \sup_{k=1,\dots,N}\|\phi^k\|_X$ and $\norm{\phi^\tau}^p_{\ell^p(X)} =\tau\sum_{k=1}^N\|\phi^k\|_X^p$ for $p\in[1,\infty)$ and the difference operators $\delta\phi^{k+1} = (\phi^{k+1}-\phi^{k})/\tau$ for $k=0,\dots, N-1$ and $\bar{\delta}\phi^{k} = -\delta\phi^{k+1}$ for $k=N-1,\dots, 0$. In addition, we let $\hat{\phi}^\tau$ denote its piecewise linear interpolation given by $\hat{\phi}^\tau(0) = \phi^0$ and $\hat{\phi}^\tau(t) = \phi^k+(t-t_k)\delta \phi^{k+1}$ for $t\in(t_k,t_{k+1}]$ and $k = 0,\dots, N-1$, and $\tilde{\phi}^\tau$ represent the piecewise constant function given by
 \begin{align}
 \label{def:piecewise_constant}
 \begin{cases}
 \tilde{\phi}^\tau(0) = \phi^0\\
     \tilde{\phi}^\tau(t) = \phi^k,\quad t\in(t_{k-1},t_k],\quad k=1,\dots, N.
 \end{cases}
 \end{align}
 
 \begin{remark}
 \label{discrete_to_continuous}
     For any sequence $\phi^\tau$, we have $\norm{\phi^\tau}_{\ell^2(X)} = \|\tilde{\phi}^\tau\|_{L_2(0,T;X)}$. Indeed,
     \begin{align*}
\textstyle\|\tilde{\phi}^\tau\|^2_{L_2(0,T;X)} = \sum_{k = 1}^N \int_{t_{k-1}}^{t_k} \|\tilde{\phi}^\tau(t)\|^2_X\dd t = \sum_{k = 1}^N \int_{t_{k-1}}^{t_k} \|\phi^k\|^2_X\dd t =\norm{\phi^\tau}^2_{\ell^2(X)}.
     \end{align*}
     We will henceforth identify $\|\tilde{\phi}^\tau\|_{L_2(0,T;X)}$ as $\norm{\phi^\tau}_{L_2(0,T;X)}$ for notational simplicity.
 \end{remark}
\subsection{Spatial discretization}
\phantomsection 
\label{sec:finite_element_approximation}
To build the finite element approximation, we define a space of continuous, piecewise linear functions, denoted by $V_h$, which is spanned by a collection of hat functions $\{\psi_h^i\}_{i=1}^{N_h}$ defined over the graph. These basis functions fall into two categories: those associated with internal nodes along the edges and those centered at graph vertices. See Figure~\ref{basis}. Each edge is subdivided into uniformly spaced segments, and standard hat functions are defined at the internal nodes. At each vertex, a special hat function is constructed with support on all incident edges, forming a star-shaped structure around the vertex. Let $\mathcal T_h$ denote the resulting partition of $\Gamma$. Throughout, we assume that $\mathcal T_h$ is conforming and quasi-uniform, with every graph vertex coinciding with a mesh node. More precisely, there exists a constant $\rho\ge1$, independent of $h$, such that $h_K\le h$ and $h\le \rho h_K$ for all $K\in\mathcal T_h$, where $h_K$ denotes the length of the mesh interval $K$ and $ h=\max_{K\in\mathcal T_h} h_K$. Further details on the construction of the basis functions can be found in \citep{arioli2018finite}.
\begin{figure}[t]
\centering
\includegraphics[width=0.5\textwidth]{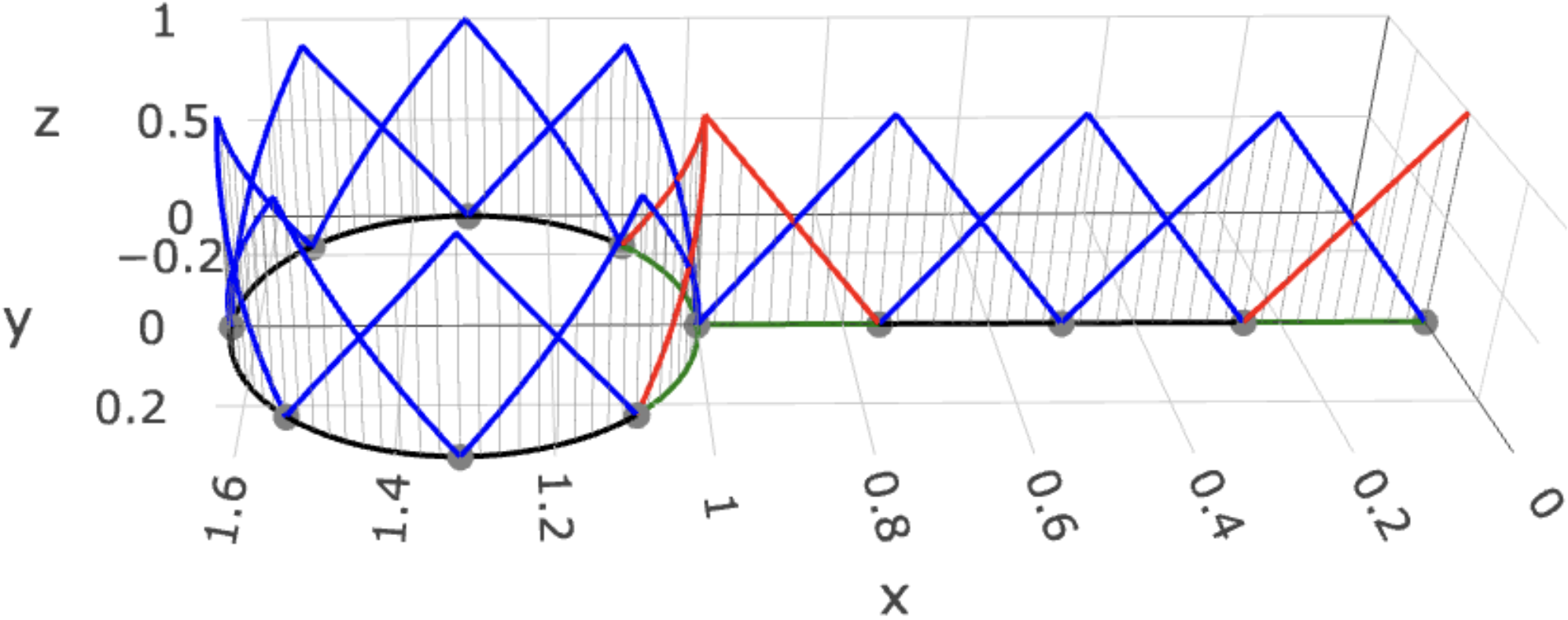}
\caption{Illustration of the basis function system $\{\psi^i_h\}_{i=1}^{N_h}$ on the tadpole graph (in black). Standard hat functions associated with internal edge nodes are shown in blue, while vertex-centered functions are highlighted in red. The star-shaped structures around the vertices are shown in green.}
\label{basis}
\end{figure}

To approximate the fractional operator in the discrete setting, we introduce the following operator acting on the finite element space $V_h$.
\begin{definition}
   We define the discrete version of $L$ as $L_h:V_h\to V_h$ via 
\begin{align*}
    (L_h\phi,\psi)_{L_2(\Gamma)} = a(\phi, \psi),\quad \phi,\psi\in V_h,
\end{align*} 
where $a:H^1(\Gamma)\times H^1(\Gamma)\to\mathbb{R}$ given by $a(u,v)= \kappa^2(u,v)_{L_2(\Gamma)} + (u',v')_{L_2(\Gamma)}$ is the bilinear form associated with $L$.
\end{definition}
Since $L_h$ is symmetric positive definite on $V_h$, it admits a spectral decomposition with eigenpairs $\{(\lambda_{j,h},e_{j,h})\}_{j=1}^{N_h}$. The eigenfunctions $\{e_{j,h}\}_{j=1}^{N_h}$ are orthonormal in $L_2(\Gamma)$, while the eigenvalues satisfy $0<\kappa^2=\lambda_{1,h}\leq\cdots\leq\lambda_{N_h,h}$ and $\lambda_j\leq\lambda_{j,h}$ for $j=1,\dots,N_h$ \citep[Section~6.3]{bolin2023regularity}. For $\alpha\in[0,2]$, we define the fractional power of $L_h$ by $L_h^{\sfrac{\alpha}{2}}\phi=\sum_{j=1}^{N_h}\lambda_{j,h}^{\sfrac{\alpha}{2}}(\phi,e_{j,h})_{L_2(\Gamma)}e_{j,h}$ for $\phi\in V_h$. This induces the discrete bilinear form $\mathfrak{a}_{h,\alpha}:V_h\times V_h\to\mathbb{R}$, given by $\mathfrak{a}_{h,\alpha}(\phi,\psi)=(L_h^{\sfrac{\alpha}{4}}\phi,L_h^{\sfrac{\alpha}{4}}\psi)_{L_2(\Gamma)}$ for $\phi,\psi\in V_h$.
\subsection{Elliptic projection}
In this subsection, we introduce the $L_2(\Gamma)$-orthogonal projector $P_h$ and the elliptic projector $G_{h,\alpha}$. While $P_h$ provides optimal approximation properties in $L_2(\Gamma)$, the operator $G_{h,\alpha}$ is defined via the bilinear form $\mathfrak{a}_{\alpha}(\cdot,\cdot)$ and plays a central role in the error analysis of the discrete fractional operator.

\begin{definition}
Let $P_h:L_2(\Gamma)\to V_h$ denote the $L_2(\Gamma)$-orthogonal projection onto $V_h$, i.e., for $\phi\in L_2(\Gamma)$, the projection $P_h\phi$ is the unique function in $V_h$ satisfying
\begin{align}
\label{eq:projector}
    (P_h\phi,v_h)_{L_2(\Gamma)} = (\phi,v_h)_{L_2(\Gamma)},\quad\forall v_h\in V_h.
\end{align}
\end{definition}
\begin{definition}
Let $\alpha\in[0,2]$. For $\alpha=0$, we define $G_{h,0}:=P_h$. For $\alpha\in(0,2]$ and $\phi\in\dot{H}^{\sfrac{\alpha}{2}}(\Gamma)$, the elliptic projector $G_{h,\alpha}:\dot{H}^{\sfrac{\alpha}{2}}(\Gamma)\to V_h$ is such that $G_{h,\alpha}\phi\in V_h$ is unique and
\begin{align}
\label{eq:elliptic_projector}
        \mathfrak{a}_{\alpha}(G_{h,\alpha}\phi,v_h) = \mathfrak{a}_{\alpha}(\phi,v_h),\quad\forall v_h\in V_h.
\end{align}

\end{definition}
\begin{proposition}
\label{proposition:with_projection}
    Let $\alpha\in[0,2]$. If $\phi\in\dot{H}^{\alpha}(\Gamma)$, then
    \begin{align}
    \label{ineq:estimateforPh}
        \norm{\phi-P_h \phi}_{L_2(\Gamma)}\lesssim h^{\alpha}\norm{\phi}_{\dot{H}^{\alpha}(\Gamma)}.
    \end{align}
\end{proposition}
\begin{proof}
    Since $P_h$ is the $L_2(\Gamma)$-orthogonal projection, we have $\norm{(I-P_h)\phi}_{L_2(\Gamma)}\leq\norm{\phi}_{L_2(\Gamma)}$. This corresponds to \eqref{ineq:estimateforPh} with $\alpha = 0$. By definition of $P_h$ and \citep[Proposition~6.2]{bolin2023regularity}, we obtain $\norm{(I-P_h)\phi}_{L_2(\Gamma)}\leq\norm{(I-G_{h,2})\phi}_{L_2(\Gamma)}\leq h^2\norm{\phi}_{\dot{H}^2(\Gamma)}$, which yields \eqref{ineq:estimateforPh} with $\alpha = 2$. Having obtained estimate \eqref{ineq:estimateforPh} for $\alpha = 0$ and $\alpha=2$, we can apply the interpolation theory from \citep[Section~3.2]{ChandlerWilde2015interpolation} (specifically,  Theorem 3.5 and (ii) on page 428 with $T=I-P_h$) to obtain estimate \eqref{ineq:estimateforPh} for $\alpha\in[0,2]$ as well. 
\end{proof}
\begin{lemma}
\label{lem:graph-fractional-fe-approx}
    If $\mathcal{T}_h$ is quasi-uniform, then for every $0\leq t\leq1$ and $t\leq s\leq2$, we have $\inf_{v_h\in V_h}\|u-v_h\|_{\dot H^t(\Gamma)}\lesssim h^{s-t}\|u\|_{\dot H^s(\Gamma)}$ for all $u\in \dot H^s(\Gamma)$. In particular, for every $\alpha\in(0,2]$, we have $\inf_{v_h\in V_h}\|u-v_h\|_{\dot H^{\sfrac{\alpha}{2}}(\Gamma)}\lesssim h^{\sfrac{\alpha}{2}}\|u\|_{\dot H^\alpha(\Gamma)}$ for all $u\in \dot H^\alpha(\Gamma)$.
\end{lemma}
\begin{proof}
    We first prove the estimate for $0\leq t\leq s\leq1$. Since $\dot H^1(\Gamma)=H^1(\Gamma)$ with equivalent norms, the spectral characterization of $L$ implies that it suffices to prove the result for the interpolation scale $H^r(\Gamma)=[L_2(\Gamma),H^1(\Gamma)]_r$, $0\le r\le1$. For each node $a\in\mathcal N_h$, fix once and for all a mesh interval $K_a\in\mathcal T_h$ having $a$ as an endpoint. If several such intervals exist, choose one arbitrarily. For each node $a\in\mathcal N_h$, let $\psi_a\in V_h$ be the associated global hat function, let $\ell_a = \psi_a|_{K_a}$ denote the local affine nodal basis function on $K_a$ corresponding to $a$, and let $\eta_a\in\mathbb P_1(K_a)$ be the local dual basis function satisfying $\int_{K_a}\eta_a\ell_b\dd x=\delta_{ab}$ for each endpoint $b$ of $K_a$. Define the linear functional $\lambda_a(u)=\int_{K_a}u\,\eta_a\dd x$ for $u\in L_2(\Gamma)$ and define the quasi-interpolant $\Pi_hu=\sum_{a\in\mathcal N_h}\lambda_a(u)\psi_a$. Then $\Pi_h:L_2(\Gamma)\to V_h$ is well-defined. Moreover, $\Pi_h$ reproduces $V_h$. Indeed, if $v_h\in V_h$, then $v_h|_{K_a}$ is affine and therefore $\lambda_a(v_h)=v_h(a)$. Since $v_h$ is continuous at graph vertices, this value is independent of the edge by which the vertex is approached. Hence $\Pi_hv_h=\sum_{a\in\mathcal N_h}v_h(a)\psi_a=v_h$. For each interval $K\in\mathcal T_h$, denote its endpoints by $a_K$ and $b_K$ and define $\omega_K=K\cup K_{a_K}\cup K_{b_K}$. This is a connected subgraph of $\Gamma$. It consists of at most three mesh intervals and satisfies $|\omega_K|\leq (1+2\rho)h_K$. Let $D_\Gamma:=\max\{2,\max_{v\in\mathcal V}\deg(v)\}$. Thus, every mesh node is contained in at most $D_\Gamma$ mesh intervals. Furthermore, each fixed mesh interval $J\in\mathcal T_h$ is contained in at most $N_{\mathrm{ov}}:=1+2D_\Gamma$ of the patches $\omega_K$.

    Next we prove $L_2$ stability. On the reference interval $(0,1)$, the two dual functions to the affine nodal basis have finite $L_2$ norms. By scaling, $\|\eta_a\|_{L_2(K_a)}\lesssim h_{K_a}^{-\sfrac{1}{2}}$ and consequently, $|\lambda_a(u)|\lesssim h_{K_a}^{-\sfrac{1}{2}}\|u\|_{L_2(K_a)}$. On any interval $K$ with endpoint $a$, the restriction of $\psi_a$ is affine and satisfies $\|\psi_a\|_{L_2(K)}\leq h_K^{\sfrac{1}{2}}$ and $\|\psi_a'\|_{L_2(K)}\leq h_K^{-\sfrac{1}{2}}$. By quasi-uniformity, $h_K/h_{K_a}\leq\rho$, whenever $a$ is an endpoint of $K$. Therefore, $\|\Pi_hu\|_{L_2(K)} \leq |\lambda_{a_K}(u)|\|\psi_{a_K}\|_{L_2(K)} + |\lambda_{b_K}(u)|\|\psi_{b_K}\|_{L_2(K)} \lesssim \rho^{\sfrac{1}{2}}\|u\|_{L_2(\omega_K)}$. After squaring and summing over $K$, and using the overlap bound
    $N_{\mathrm{ov}}$, we obtain $\|\Pi_hu\|_{L_2(\Gamma)}\lesssim\|u\|_{L_2(\Gamma)}$.

    Next we prove $H^1$ stability. Since $\Pi_h$ reproduces constants, for any constant $c$, we have $(\Pi_hu)'=(\Pi_h(u-c))'$. Using the derivative estimate above gives $\|(\Pi_hu)'\|_{L_2(K)}\lesssim h_K^{-1}\|u-c\|_{L_2(\omega_K)}$. Let $c_{\mathrm{av}}$ denote the average of $u$ over $\omega_K$. Since $\omega_K$ is connected, the Poincare inequality on a connected one-dimensional metric graph gives $ \|u-c_{\mathrm{av}}\|_{L_2(\omega_K)} \leq|\omega_K|\,\|u'\|_{L_2(\omega_K)}$. Indeed, for almost every $x,y\in\omega_K$, continuity on the connected patch and integration along a path in $\omega_K$ joining $x$ and $y$ imply $|u(x)-u(y)|\leq|\omega_K|^{\sfrac{1}{2}}\|u'\|_{L_2(\omega_K)}$, and the above Poincare inequality follows by averaging in $y$ and then integrating in $x$. Therefore, $\|(\Pi_hu)'\|_{L_2(K)}\lesssim h_K^{-1}|\omega_K|\|u'\|_{L_2(\omega_K)}\lesssim(1+2\rho)\|u'\|_{L_2(\omega_K)}$. Squaring and summing over $K$, again using the overlap bound $N_{\mathrm{ov}}$, yields $\|(\Pi_hu)'\|_{L_2(\Gamma)}\lesssim\|u'\|_{L_2(\Gamma)}$. Together with the $L_2$ stability, this gives $\|\Pi_hu\|_{H^1(\Gamma)}\lesssim\|u\|_{H^1(\Gamma)}$.

    Let $E_h=I-\Pi_h$. Since $\Pi_h$ reproduces constants, $E_hu=(u-c_{\mathrm{av}})-\Pi_h(u-c_{\mathrm{av}})$. Hence, by the local $L_2$ stability of $\Pi_h$ and the Poincare inequality on $\omega_K$, we have $\|E_hu\|_{L_2(K)}\lesssim \|u-c_{\mathrm{av}}\|_{L_2(\omega_K)}\lesssim h_K\|u'\|_{L_2(\omega_K)}$.  Squaring, summing over $K$, and using the overlap bound gives $\|E_hu\|_{L_2(\Gamma)}\lesssim h\|u\|_{H^1(\Gamma)}$. By the $L_2$- and $H^1$-stability of $\Pi_h$, we also have that $E_h:L_2(\Gamma)\to L_2(\Gamma)$ and $E_h: H^1(\Gamma)\to H^1(\Gamma)$ are bounded. By interpolation in the target space, for every $0\leq t\leq1$, we have $\|E_hu\|_{H^t(\Gamma)}\lesssim h^{1-t}\|u\|_{H^1(\Gamma)}$. Also, interpolation of the $L_2$ and $H^1$ stability estimates gives $\|E_hu\|_{H^t(\Gamma)}\lesssim\|u\|_{H^t(\Gamma)}$. Fix $0\leq t\leq s\leq1$. If $s=t$, the desired bound follows from the last estimate. If $t<s\leq1$, interpolate the two mappings $E_h:H^t(\Gamma)\to H^t(\Gamma)$ and $E_h:H^1(\Gamma)\to H^t(\Gamma)$. Since $H^s(\Gamma)=[H^t(\Gamma),H^1(\Gamma)]_\theta$ for $\theta=(s-t)/(1-t)$, we obtain $\|E_hu\|_{H^t(\Gamma)}\lesssim h^{\theta(1-t)}\|u\|_{H^s(\Gamma)}= h^{s-t}\|u\|_{H^s(\Gamma)}$. Returning to the equivalent spectral norms gives $\|u-\Pi_hu\|_{\dot H^t(\Gamma)}\lesssim h^{s-t}\|u\|_{\dot H^s(\Gamma)}$ for $0\leq t\leq s\leq1$.

    It remains to treat the case $1\leq s\leq2$ and $0\leq t\leq1$. In this range we use the nodal interpolant $I_hu\in V_h$. This is well-defined because $\dot H^s(\Gamma)\subset H^1(\Gamma)$ for $s\geq1$, and functions in $H^1(\Gamma)$ are continuous on the metric graph. Let $F_h = I-I_h$. On each mesh interval $K=[x_0,x_1]$, the restriction of $I_hu$ is the affine interpolant of $u|_K$ at the endpoints. For $u\in H^1(\Gamma)$, we have $(I_hu)'|_K = (u(x_1)-u(x_0))/h_K= h_K^{-1}\int_K u'(x)\dd x$. Hence, $\|(I_hu)'\|_{L_2(K)}\leq\|u'\|_{L_2(K)}$ and $ \|(F_hu)'\|_{L_2(K)}\lesssim\|u'\|_{L_2(K)}$. Moreover, since $F_hu$ vanishes at the endpoints of $K$, we have $\|F_hu\|_{L_2(K)}\leq h_K\|(F_hu)'\|_{L_2(K)}\lesssim h_K\|u'\|_{L_2(K)}$. After summing over $K$, we obtain $\|F_hu\|_{L_2(\Gamma)}\lesssim h\|u\|_{\dot H^1(\Gamma)}$ and  $\|F_hu\|_{\dot H^1(\Gamma)}\lesssim\|u\|_{\dot H^1(\Gamma)}$. For $u\in\dot H^2(\Gamma)=D(L)$, we have $u|_e\in H^2(e)$ on each edge $e$ and $        \sum_{e}\|u_e''\|_{L_2(e)}^2\lesssim\|u\|_{\dot H^2(\Gamma)}^2$. Indeed, edgewise, $Lu=\kappa^2u-u''$, and hence $\|u''\|_{L_2(\Gamma)}\leq\|Lu\|_{L_2(\Gamma)}+\kappa^2\|u\|_{L_2(\Gamma)}\lesssim\|u\|_{\dot H^2(\Gamma)}$. The standard one-dimensional interpolation estimates on each mesh interval give $\|F_hu\|_{L_2(K)}\lesssim h_K^2\|u''\|_{L_2(K)}$ and $\|(F_hu)'\|_{L_2(K)}\lesssim h_K\|u''\|_{L_2(K)}$. Summing over $K$ yields $\|F_hu\|_{L_2(\Gamma)}\lesssim h^2\|u\|_{\dot H^2(\Gamma)}$ and $\|F_hu\|_{\dot H^1(\Gamma)}\lesssim h\|u\|_{\dot H^2(\Gamma)}$. Interpolating between the endpoint estimates $I-I_h:\dot H^1(\Gamma)\to L_2(\Gamma)$ and $I-I_h:\dot H^2(\Gamma)\to L_2(\Gamma)$ gives $\|F_hu\|_{L_2(\Gamma)}\lesssim h^s\|u\|_{\dot H^s(\Gamma)} $ for $1\leq s\leq2$. Similarly, interpolating $I-I_h:\dot H^1(\Gamma)\to \dot H^1(\Gamma)$ and $I-I_h:\dot H^2(\Gamma)\to \dot H^1(\Gamma)$ gives $\|F_hu\|_{\dot H^1(\Gamma)}\lesssim h^{s-1}\|u\|_{\dot H^s(\Gamma)}$ for $1\leq s\leq2$. Finally, for $0\leq t\leq1$, the interpolation inequality in the Hilbert scale generated by $L$ gives 
    \begin{align*}
        \|F_hu\|_{\dot H^t(\Gamma)}
        \lesssim
        \|F_hu\|_{L_2(\Gamma)}^{1-t}
        \|F_hu\|_{\dot H^1(\Gamma)}^t  \lesssim h^{s-t}\|u\|_{\dot H^s(\Gamma)} .
    \end{align*}
    Since $I_hu\in V_h$, this also proves the desired best-approximation estimate for $1\leq s\leq2$.

\end{proof}
\begin{proposition}
\label{proposition:with_elliptic_projection}
Let $\alpha\in[0,2]$. If $\phi\in\dot{H}^{\alpha}(\Gamma)$, then
     \begin{align}
     \label{ineq:estimateforha}
        \norm{\phi - G_{h,\alpha}\phi}_{L_2(\Gamma)} \lesssim h^\alpha\norm{\phi}_{\dot{H}^{\alpha}(\Gamma)}.
    \end{align}
\end{proposition}
\begin{proof}
    Let $\alpha\in(0,2]$ and $e_h=\phi-G_{h,\alpha}\phi$. By using \eqref{eq:elliptic_projector}, the Galerkin orthogonality relation $\mathfrak{a}_{\alpha}(e_h,v_h)=0$ holds for every $v_h\in V_h$. Since $\mathfrak{a}_{\alpha}$ is the inner product on $\dot H^{\sfrac{\alpha}{2}}(\Gamma)$, $G_{h,\alpha}$ is the $\mathfrak{a}_{\alpha}$-orthogonal projection onto $V_h$. Therefore, for any $v_h\in V_h$, writing $e_h=(\phi-v_h)+(v_h-G_{h,\alpha}\phi)$ and observing that $v_h-G_{h,\alpha}\phi\in V_h$, we have $\|e_h\|_{\dot H^{\sfrac{\alpha}{2}}(\Gamma)}^2
        =\mathfrak{a}_{\alpha}(e_h,e_h)
        =\mathfrak{a}_{\alpha}(e_h,\phi-v_h)
        \le
        \|e_h\|_{\dot H^{\sfrac{\alpha}{2}}(\Gamma)}
        \|\phi-v_h\|_{\dot H^{\sfrac{\alpha}{2}}(\Gamma)}$. This and Lemma~\ref{lem:graph-fractional-fe-approx} imply
    \begin{align}
    \label{ineq:energy_estimate}
    \textstyle
        \|e_h\|_{\dot H^{\sfrac{\alpha}{2}}(\Gamma)}\leq\inf_{v_h\in V_h}\|\phi-v_h\|_{\dot H^{\sfrac{\alpha}{2}}(\Gamma)}\lesssim h^{\sfrac{\alpha}{2}}\|\phi\|_{\dot H^\alpha(\Gamma)}.
    \end{align}
    We now use a duality argument. Let $z\in\dot H^\alpha(\Gamma)$ solve $\mathfrak{a}_{\alpha}(z,v)=(e_h,v)_{L_2(\Gamma)}$ for all $v\in \dot H^{\sfrac{\alpha}{2}}(\Gamma)$. By elliptic regularity (Proposition \ref{proposition:elliptic_regularity} with $r=0$, $u = z$, and $f=e_h$), we have $\|z\|_{\dot H^\alpha(\Gamma)} = \|e_h\|_{L_2(\Gamma)}$. Taking $v=e_h$ and letting $z_h\in V_h$ be arbitrary, the symmetry of $\mathfrak{a}_{\alpha}$ and Galerkin orthogonality yield $\|e_h\|_{L_2(\Gamma)}^2 = \mathfrak{a}_{\alpha}(e_h,z-z_h)\leq\|e_h\|_{\dot H^{\sfrac{\alpha}{2}}(\Gamma)}\|z-z_h\|_{\dot H^{\sfrac{\alpha}{2}}(\Gamma)}$. This and Lemma~\ref{lem:graph-fractional-fe-approx} imply
    \begin{align*}
    \textstyle
        \|e_h\|_{L_2(\Gamma)}^2 \leq \|e_h\|_{\dot H^{\sfrac{\alpha}{2}}(\Gamma)} \inf_{z_h\in V_h}\|z-z_h\|_{\dot H^{\sfrac{\alpha}{2}}(\Gamma)} \lesssim h^{\sfrac{\alpha}{2}} \|e_h\|_{\dot H^{\sfrac{\alpha}{2}}(\Gamma)} \|z\|_{\dot H^\alpha(\Gamma)} 
    \end{align*}
    By \eqref{ineq:energy_estimate} and the regularity estimate for $z$, we get $\|e_h\|_{L_2(\Gamma)}^2\lesssim h^\alpha \|\phi\|_{\dot H^\alpha(\Gamma)}\|e_h\|_{L_2(\Gamma)}$. If $e_h\neq0$, division by $\|e_h\|_{L_2(\Gamma)}$ yields \eqref{ineq:estimateforha}. If $e_h=0$, the estimate is immediate. Finally, since $G_{h,0}=P_h$ by convention, then for $\alpha=0$ the
estimate becomes $\|\phi-P_h\phi\|_{L_2(\Gamma)}\leq\|\phi\|_{L_2(\Gamma)}$, which is the stability of the $L_2(\Gamma)$-orthogonal projection.
\end{proof}
\begin{lemma}
    \label{lemma:boundwithhminus2}
    If $\mathcal{T}_h$ is quasi-uniform, then $\|L_h\phi\|_{L_2(\Gamma)}\lesssim h^{-2}\|\phi\|_{L_2(\Gamma)}$ for $\phi\in V_h$.
\end{lemma}
\begin{proof}
    Since $L_h \phi \in V_h$, it follows that  $\left\|L_h \phi\right\|_{L_2(\Gamma)}=\sup _{\substack{\psi \in V_h :\|\psi\|_{L_2(\Gamma)}=1}} a(\phi, \psi)$. Under quasi-uniformity, the inverse inequality gives
        \begin{align*}
            a(\phi, \psi) \leq\|\phi\|_{\dot{H}^1(\Gamma)}\|\psi\|_{\dot{H}^1(\Gamma)} \lesssim(h^{-1}\|\phi\|_{L_2(\Gamma)})(h^{-1}\|\psi\|_{L_2(\Gamma)})=h^{-2}\|\phi\|_{L_2(\Gamma)}.
        \end{align*}
        Taking the supremum gives the result.
\end{proof}
\begin{proposition}
\label{prop:controlofLhbyL}
        Let $\alpha\in[0,2]$. If $\phi\in\dot{H}^{\alpha}(\Gamma)$, then
        \begin{align}
        \label{ineq:controlofLhbyL}
            \|L_h^{\sfrac{\alpha}{2}}P_h\phi\|_{L_2(\Gamma)} \lesssim \|L^{\sfrac{\alpha}{2}}\phi\|_{L_2(\Gamma)}.
        \end{align}
\end{proposition}
\begin{proof}
     For $\alpha = 0$, the result is immediate. For $\alpha = 2$, let $\phi\in D(L) = \dot{H}^2(\Gamma)$ and $v_h\in V_h$. It follows that $(L_hG_{h,2}\phi,v_h)_{L_2(\Gamma)} = \mathfrak{a}_{2}(G_{h,2}\phi,v_h)  = \mathfrak{a}_{2}(\phi,v_h)  = (L\phi,v_h)_{L_2(\Gamma)} = (P_hL\phi,v_h)_{L_2(\Gamma)}$. Since this holds for all $v_h\in V_h$, we have that $L_hG_{h,2}\phi = P_hL\phi$ for all $\phi\in \dot{H}^2(\Gamma)$. This and the contraction property of $P_h$ yields $\|L_hP_h\phi\|_{L_2(\Gamma)} \leq \|L\phi\|_{L_2(\Gamma)}  +  \|L_h(P_h\phi - G_{h,2}\phi)\|_{L_2(\Gamma)}$. Since $P_h\phi - G_{h,2}\phi\in V_h$, Lemma \ref{lemma:boundwithhminus2} implies that $\|L_h(P_h\phi - G_{h,2}\phi)\|_{L_2(\Gamma)}  \lesssim h^{-2}\|P_h\phi - G_{h,2}\phi\|_{L_2(\Gamma)}$. By Propositions \ref{proposition:with_projection} and \ref{proposition:with_elliptic_projection}, we have that $\|P_h\phi - G_{h,2}\phi\|_{L_2(\Gamma)}  \lesssim h^2\|L\phi\|_{L_2(\Gamma)} + h^2\|L\phi\|_{L_2(\Gamma)}$. By combining the previous three inequalities, we obtain \eqref{ineq:controlofLhbyL} with $\alpha=2$. It remains to consider $0<\alpha<2$. Let $\beta=\sfrac{\alpha}{2}$. We use interpolation between the cases $\alpha=0$ and $\alpha=2$. The two endpoint estimates show that $P_h:(L_2(\Gamma),\|\cdot\|_{L_2(\Gamma)})\to (V_h,\|\cdot\|_{L_2(\Gamma)})$ and $P_h:(D(L),\|L\cdot\|_{L_2(\Gamma)})\to (V_h,\|L_h\cdot\|_{L_2(\Gamma)})$ are bounded uniformly in $h$. Therefore, $P_h$ is a bounded couple map (see \citep{ChandlerWilde2015interpolation}), and for every $0<\beta<1$, $P_h:[L_2(\Gamma),D(L)]_\beta\to[V_h,V_h]_\beta$ is bounded uniformly in $h$. Since $L$ and $L_h$ are positive self-adjoint operators, \citep[Theorem~4.36]{Lunardi2018Interpolation} yields $[L_2(\Gamma),D(L)]_\beta=D(L^\beta)$ and $[V_h,V_h]_\beta=D(L_h^\beta)=V_h$, where the interpolation norms are equivalent to the corresponding graph norms. Moreover, since the spectra of $L$ and $L_h$ are bounded below by $\kappa^2>0$, these graph norms are equivalent to $\|L^\beta\cdot\|_{L_2(\Gamma)}$ and $\|L_h^\beta\cdot\|_{L_2(\Gamma)}$, respectively. Consequently, for every $\phi\in D(L^\beta)$, $\|L_h^\beta P_h\phi\|_{L_2(\Gamma)}\lesssim\|L^\beta\phi\|_{L_2(\Gamma)}$. 
    \end{proof}

We now introduce both semidiscrete and fully discrete schemes for approximating the solution to problem~\eqref{weakformofentireequation}.
Let $\alpha>0$ and $U^\tau$ denote the sequence of time-discrete approximations of the solution to problem \eqref{weakformofentireequation}. The scheme~\eqref{system:semidiscrete_scheme}, referred to as the \emph{semidiscrete scheme}, determines the sequence $U^\tau$ of elements of $\dot{H}^{\sfrac{\alpha}{2}}(\Gamma)$ by requiring that the iterate $U^{k+1}$ satisfies
\begin{equation}
\label{system:semidiscrete_scheme}
    \left\{
\begin{aligned}
    \langle\delta U^{k+1},\phi\rangle + \mathfrak{a}_{\alpha}(U^{k+1},\phi)  &= \langle f^{k+1},\phi\rangle , \quad\forall\phi\in \dot{H}^{\sfrac{\alpha}{2}}(\Gamma),\quad k=0,\dots, N-1, \\
    U^0 &= u_0,
\end{aligned}
\right.
\end{equation}
where $f^{k+1} = \dfrac{1}{\tau}\int_{t_k}^{t_{k+1}}f(t)\dd t$.

We now extend the construction to include spatial discretization. Let $\alpha\in(0,2]$ and $U_h^\tau$ denote the sequence of approximations of the solution to problem \eqref{weakformofentireequation} at each time step on a mesh indexed by $h$. The scheme~\eqref{system:fully_discrete_scheme}, referred to as the \emph{fully discrete scheme}, determines $U_h^{\tau}\subset V_h$ by requiring that the iterate $U^{k+1}_h$ satisfies
\begin{equation}
\label{system:fully_discrete_scheme}
    \left\{
\begin{aligned}
    \langle\delta U_h^{k+1},\phi\rangle + \mathfrak{a}_{h,\alpha}(
        U_h^{k+1},\phi) &= \langle f^{k+1},\phi\rangle, \quad\forall\phi\in V_h,\quad k=0,\dots, N-1, \\
    U^0_h &= P_hu_0.
\end{aligned}
\right.
\end{equation}
\section{Time and space error analysis}
\label{timeandspaceerroranalysis}
In this section, we derive rigorous error estimates that quantify the accuracy of the numerical scheme with respect to both temporal and spatial discretizations.
\begin{theorem}
\label{theorem:first_part_of_main}
    Let $\alpha>0$ and $u$ be the solution to \eqref{weakformofentireequation}. Let $\tilde{U}^\tau$ be the piecewise constant function (defined via \eqref{def:piecewise_constant}) constructed from the solution to \eqref{system:semidiscrete_scheme}. If $f\in L_\infty(0,T;L_2(\Gamma))$ and $u_0\in\dot{H}^{\sfrac{\alpha}{2}}(\Gamma)$, then $\|u-\tilde{U}^\tau\|_{L_2(0,T;L_2(\Gamma))}\lesssim \tau$.
\end{theorem}
\begin{proof}
The proof follows the same arguments as in \citep[Theorem~5.6]{Glusa2021errorestimates}, with the obvious adaptations to the metric graph setting and the corresponding variational framework on $\Gamma$. We therefore omit the details.
\end{proof}
\begin{theorem}
\label{theorem:theorem_second_part_of_main}
Let $\alpha\in(0,2]$ and $U^\tau$ and $U_h^\tau$ be the solutions to \eqref{system:semidiscrete_scheme} and \eqref{system:fully_discrete_scheme}, respectively. Suppose that $u_0\in\dot{H}^{\alpha}(\Gamma)$. Then, for every sufficiently small $\epsilon>0$, we have $\norm{U^\tau-U_h^\tau}_{\ell^2(L_2(\Gamma))}\lesssim h^{\alpha-\epsilon}K_\tau$, where $K_\tau = |\ln(\tau)|^2$ if $f\in L_2(0,T;L_2(\Gamma))$ and $K_\tau = |\ln(\tau)|$ if $f\in L_2(0,T;\dot{H}^{\alpha}(\Gamma))$.
\end{theorem}
\begin{proof}
    Using the regularity of $u$ from Theorem \ref{theorem:existenceuniquenessandregularity}, the semidiscrete scheme \eqref{system:semidiscrete_scheme} reduces to $(w,\phi)_{L_2(\Gamma)}=0$ for all $\phi\in\dot{H}^{\sfrac{\alpha}{2}}(\Gamma)$, where $w = \delta U^{k+1}+L^{\sfrac{\alpha}{2}}U^{k+1}-f^{k+1}\in L_2(\Gamma)$. Since $\dot{H}^{\sfrac{\alpha}{2}}(\Gamma)$ is dense in $L_2(\Gamma)$, we conclude $w=0$ in $L_2(\Gamma)$. Thus both schemes \eqref{system:semidiscrete_scheme} and \eqref{system:fully_discrete_scheme} admit strong forms in $L_2(\Gamma)$, given by
        \begin{align}
          (L^{-\beta}+\tau  I)U^{k+1}&=L^{-\beta}(U^{k}+\tau f^{k+1}),\notag\\
          (L_h^{-\beta}+\tau  I_{V_h})U_h^{k+1}&=L_h^{-\beta}(U_h^{k}+\tau P_hf^{k+1}), \label{weak_form_m2}
    \end{align}
    respectively, which after iterating yields
    \begin{align*}
            U^{k+1}&=G^{k+1}(A)u_0+\tau\textstyle\sum_{i=0}^kG^{k+1-i}(A)f^{i+1},\quad k = 0,\dots, N-1,\\
            U_h^{k+1}&=G^{k+1}(A_h)P_h u_0+\tau\textstyle\sum_{i=0}^kG^{k+1-i}(A_h)P_hf^{i+1},\quad k = 0,\dots, N-1,
    \end{align*}
    where $\beta = \sfrac{\alpha}{2}$, $G(x) = x/(x+\tau)$, $A = L^{-\beta}$, and $A_h = L_h^{-\beta}$. Plugging $U^{k+1}$ and $U_{h}^{k+1}$ into $\norm{U^\tau-U_h^\tau}^2_{\ell^2(L_2(\Gamma))} = \tau\sum_{k=0}^{N-1}\norm{U^{k+1}-U_{h}^{k+1}}^2_{L_2(\Gamma)}$, using that $G^{k}(A_h)P_h \phi = (G(A_h)P_h)^{k} \phi$ for all $k\in\mathbb{N}$ and $\phi\in L_2(\Gamma)$, and applying the inequality $(a+b)^2\le2a^2+2b^2$, we obtain
    \begin{align*}
        \norm{U^\tau-U_{h}^\tau}^2_{\ell^2(L_2(\Gamma))} &\leq 2\tau\textstyle\sum_{k=0}^{N-1}\norm{(G^{k+1}(A)-(G(A_h)P_h)^{k+1})u_0}^2_{L_2(\Gamma)}\\
        &+2\tau\textstyle\sum_{k=0}^{N-1}\norm{\tau\sum_{i=0}^k(G^{k+1-i}(A)-(G(A_h)P_h)^{k+1-i})f^{i+1}}^2_{L_2(\Gamma)}\\
        &=:E^2_1+E^2_2.
    \end{align*}
     Let $G_h =G(A_h)P_h$ and $\phi\in L_2(\Gamma)$. Using the telescoping identity $X^{n+1}-Y^{n+1} = \sum_{j=0}^{n} X^{n-j}(X-Y)Y^j$ with $X =G(A)$ and $Y = G_h$ yields
    \begin{align}
    \label{eq:diffGkplus1AminusGAhPhKplus1}
    \textstyle
        (G^{n+1}(A)-G_h^{n+1})\phi = \sum_{j=0}^{n} G^{n-j}(A)(G(A) - G_h)G^{j}(A_h)P_h\phi.
    \end{align}
    Let $R = (A+\tau I)^{-1}$ and $R_h = (A_h+\tau I)^{-1}$. Since $G(A) = AR = I-\tau R$ and $G(A_h) = A_hR_h =  I-\tau R_h$, it follows that $G(A) - G_h = (I-P_h) + \tau(R_hP_h-R) = (I-P_h) -\tau R_h(I-P_h) + \tau(R_h-R)$. Consequently,
    \begin{align}
    \label{eq:diffGAminusGAhPh}
        (G(A)\! - \!G_h) G^{j}(A_h)P_h\phi = \tau(R_h\!-\!R) G^{j}(A_h)P_h\phi = \tau R(A\!-\!A_h)R_h G^{j}(A_h)P_h\phi,
    \end{align}
    where the terms involving $(I-P_h)$ vanish since $G^{j}(A_h)P_h\phi\in V_h$, and the final equality follows from the resolvent identity. Replacing $A-A_h = (L^{-\beta} - L_h^{-\beta}P_h)-L_h^{-\beta}(I-P_h)$ into \eqref{eq:diffGAminusGAhPh} yields $(G(A) - G_h) G^{j}(A_h)P_h\phi = \tau R(L^{-\beta} - L_h^{-\beta}P_h)R_h G^{j}(A_h)P_h\phi$, where the term involving $(I-P_h)$ vanishes because $R_hG^{j}(A_h)P_h\phi\in V_h$. Therefore,
    \begin{align*}
        \textstyle
        (G^{n+1}(A)-G_h^{n+1})\phi  = \tau\sum_{j=0}^{n} G^{n-j}(A)R(L^{-\beta} - L_h^{-\beta}P_h)R_h G^{j}(A_h)P_h\phi.
\end{align*}
By \citep[Lemma~6.4(a)]{bolin2023regularity}, taking $s=0$ in their notation, we obtain
\begin{align}
\label{eq:difflastbeforenorm}
    \textstyle
        \|(G^{n+1}(A)-G_h^{n+1})\phi\|_{L_2(\Gamma)}  \!\leq\! \tau h^{\alpha-\epsilon}\sum_{j=0}^{n}\|G^{n-j}(A)R\|_{\mathcal L(L_2(\Gamma))}\|R_h G^{j}(A_h)P_h\phi\|_{L_2(\Gamma)}
\end{align}
for all sufficiently small $\epsilon>0$. Observe that $G^{n-j}(A)R = L^\beta(I+\tau L^\beta)^{-(n-j+1)}$ and $R_hG^j(A_h) = (I+\tau L_h^\beta)^{-(j+1)}L_h^\beta$. Motivated by these identities, define $\varphi_m(x)= x^\beta(1+\tau x^\beta)^{-(m+1)}$, where $m\in\mathbb{N}_0$ and $x\geq0$. For $m\in\mathbb{N}$, this function attains its maximum value $\varphi_m(x^*) = \tau^{-1}p_m$ at $x^* = (m\tau)^{-\sfrac{1}{\beta}}$, where $p_m:=m^m(m+1)^{-(m+1)}$ satisfies $p_m\leq(m+1)^{-1}$.

\paragraph{Estimate for $E_1^2$} Let $d_u=\|(G^{k+1}(A)-G_h^{k+1})u_0\|_{L_2(\Gamma)}$. Since $R_h G^{j}(A_h)=(I+\tau L_h^{\beta})^{-(j+1)}L_h^{\beta}$, $\sup_{\lambda\in\sigma(L_h)}(1+\tau \lambda^{\beta})^{-(j+1)}\leq 1$, and Proposition \ref{prop:controlofLhbyL} applies, 
    \begin{align*}
        \|R_h G^{j}(A_h)P_hu_0\|_{L_2(\Gamma)} \leq \|L_h^{\beta}P_hu_0\|_{L_2(\Gamma)}\lesssim \|u_0\|_{\dot{H}^{2\beta}(\Gamma)}.
    \end{align*} 
    By taking $n=k$ and $\phi=u_0$ in \eqref{eq:difflastbeforenorm}, and using the above inequality, we obtain $d_u \leq \tau h^{\alpha-\epsilon}\|u_0\|_{\dot{H}^{2\beta}(\Gamma)}\sum_{j=0}^{k}\tau^{-1}p_{k-j}$. Since $\sum_{j=0}^{k}p_{k-j}\leq \sum_{j=0}^{k}(k-j+1)^{-1} = \sum_{\ell=1}^{k+1}\ell^{-1}\leq 1+\ln(k+1)$, we get $d_u \leq  h^{\alpha-\epsilon}\|u_0\|_{\dot{H}^{2\beta}(\Gamma)}(1+\ln(k+1))$. Therefore, 
    \begin{align}
    \label{eq:v1}
    \textstyle
        E_1^2  \leq 2\tau  h^{2(\alpha-\epsilon)}\|u_0\|^2_{\dot{H}^{2\beta}(\Gamma)}N(1+\ln(N))^2\lesssim  h^{2(\alpha-\epsilon)}|\ln(\tau)|^2\|u_0\|^2_{\dot{H}^{2\beta}(\Gamma)}.
    \end{align}

\paragraph{Estimate for $E_2^2$}
    For simplicity, let $e_{k+1} = \tau\sum_{i=0}^kd_i$, where $d_i = (G^{k+1-i}(A)-G_h^{k+1-i})f^{i+1}$. Setting $n=k-i$ and $\phi = f^{i+1}$ in \eqref{eq:difflastbeforenorm}, and using the standard logarithmic estimate for the harmonic numbers, yields
    \begin{align*}
    \textstyle
    \|d_i\|_{L_2(\Gamma)} &\leq \tau h^{\alpha-\epsilon}\|f^{i+1}\|_{L_2(\Gamma)}\textstyle\sum_{j=0}^{k-i} \tau^{-2}p_{k-i-j}p_j \\
    &\leq\tau^{-1}h^{\alpha-\epsilon}\|f^{i+1}\|_{L_2(\Gamma)}(1+\ln(k-i+1))(k-i+1)^{-1}.
    \end{align*}
    Since $(b*c)_k = \sum_{i\in\mathbb{Z}}b_ic_{k-i}$ with $b_m = \|f^{m+1}\|_{L_2(\Gamma)}$ and $c_m = (1+\ln(m+1))(m+1)^{-1}$ (after extending by zero), we obtain $\|e_{k+1}\|_{L_2(\Gamma)}\leq   h^{\alpha-\epsilon}(b*c)_k$. This and the discrete Young's inequality gives $E_2^2\leq 2\tau h^{2(\alpha-\epsilon)}\sum_{k=0}^{N-1}(b*c)^2_k \leq 2\tau h^{2(\alpha-\epsilon)}\|b*c\|^2_{\ell^2(\mathbb{Z})} \leq 2\tau h^{2(\alpha-\epsilon)} \|b\|^2_{\ell^2(\mathbb{Z})}\|c\|^2_{\ell^1(\mathbb{Z})}$. Since $\|b\|^2_{\ell^2(\mathbb{Z})} =\tau^{-1}\|f^\tau\|^2_{\ell^2(L_2(\Gamma))}$ and $\|c\|^2_{\ell^1(\mathbb{Z})} \leq (1+\ln(N))^4$, we finally arrive at
    \begin{align}
    \label{eq:v2}
        E_2^2\leq2\tau h^{2(\alpha-\epsilon)} \tau^{-1}\|f^\tau\|^2_{\ell^2(L_2(\Gamma))} (1+\ln(N))^4 \lesssim h^{2(\alpha-\epsilon)} \|f^\tau\|^2_{\ell^2(L_2(\Gamma))} |\ln(\tau)|^4.
    \end{align}
    Combining \eqref{eq:v1} and \eqref{eq:v2}, we obtain
    \begin{align*}
        \norm{U^\tau-U_{h}^\tau}_{\ell^2(L_2(\Gamma))}\lesssim h^{\alpha-\epsilon}|\ln(\tau)|^2\left(\norm{u_0}_{\dot{H}^{\alpha}(\Gamma)} + \norm{f^\tau}_{\ell^2(L_2(\Gamma))}\right).
    \end{align*}
    If $f\in L_2(0,T,\dot{H}^{\alpha}(\Gamma))$, then 
    \begin{align*}
    \textstyle
        \|d_i\|_{L_2(\Gamma)} \leq h^{\alpha-\epsilon}\|f^{i+1}\|_{\dot{H}^{2\beta}(\Gamma)}\sum_{j=0}^{k-i}p_{k-i-j} \leq h^{\alpha-\epsilon}\|f^{i+1}\|_{\dot{H}^{2\beta}(\Gamma)}(1+\ln(k-i+1)).
    \end{align*}
    and $\|e_{k+1}\|_{L_2(\Gamma)}\leq \tau h^{\alpha-\epsilon}(b*c)_k$, where now $b_m = \|f^{m+1}\|_{\dot{H}^{2\beta}(\Gamma)}$ and $c_m = 1+\ln(m+1)$. Since $\|b\|^2_{\ell^2(\mathbb{Z})} =\tau^{-1}\|f^\tau\|^2_{\ell^2(\dot{H}^{2\beta}(\Gamma))}$ and $\|c\|_{\ell^1(\mathbb{Z})} = \sum_{k=0}^{N-1}(1+\ln(k+1))\leq N(1+\ln(N))$, we obtain
    \begin{align}
    \label{eq:v3}
        E_2^2\leq  2\tau h^{2(\alpha-\epsilon)}\tau^{2}\|b\|^2_{\ell^2(\mathbb{Z})}\|c\|^2_{\ell^1(\mathbb{Z})}\lesssim h^{2(\alpha-\epsilon)}\|f^\tau\|^2_{\ell^2(\dot{H}^{2\beta}(\Gamma))}|\ln(\tau)|^2.
    \end{align}
    Combining \eqref{eq:v1} and \eqref{eq:v3}, we obtain
    \begin{align*}
        \norm{U^\tau-U_{h}^\tau}_{\ell^2(L_2(\Gamma))}\lesssim h^{\alpha-\epsilon}|\ln(\tau)|\left(\norm{u_0}_{\dot{H}^{\alpha}(\Gamma)} + \norm{f^\tau}_{\ell^2(\dot{H}^{\alpha}(\Gamma))}\right).
    \end{align*}
    \end{proof}
\begin{theorem}
\label{theorem:theorem_about_error_bounded_by_tau_plus_h_alpha}
Let $\alpha\in(0,2]$ and $u$ and $U_h^\tau$ be the solutions to \eqref{weakformofentireequation} and \eqref{system:fully_discrete_scheme}, respectively. Suppose that $u_0\in\dot{H}^{\alpha}(\Gamma)$. Then, for every sufficiently small $\epsilon>0$, we have $\norm{u-U_h^\tau}_{L_2(0,T;L_2(\Gamma))}\lesssim \tau+h^{\alpha-\epsilon}K_\tau$, where $K_\tau = |\ln(\tau)|^2$ if $f\in L_\infty(0,T;L_2(\Gamma))$ and $K_\tau = |\ln(\tau)|$ if $f\in L_\infty(0,T;\dot{H}^{\alpha}(\Gamma))$.
\end{theorem}
\begin{proof}
    By triangle inequality and Remark \ref{discrete_to_continuous}, we have that
    \begin{align*}
        \|u-U_h^\tau\|_{L_2(0,T;L_2(\Gamma))} &\leq \|u-\tilde{U}^\tau\|_{L_2(0,T;L_2(\Gamma))} + \|\tilde{U}^\tau-U_h^\tau\|_{L_2(0,T;L_2(\Gamma))} \\
        &= \|u-\tilde{U}^\tau\|_{L_2(0,T;L_2(\Gamma))} + \|U^\tau-U_h^\tau\|_{\ell^2(L_2(\Gamma))}. 
    \end{align*}
    Theorem \ref{theorem:first_part_of_main} and Theorem \ref{theorem:theorem_second_part_of_main} yield the desired result.
\end{proof}
\section{Rational approximation}
\label{rationalapproximationsection}
We now consider a rational approximation of $(L_h/\kappa^2)^{-\beta}$ of the form 
\begin{align}
\label{rmdef}
    r_m((L_h/\kappa^2)^{-1})= p_\ell(L_h/\kappa^2)^{-1}p_r(L_h/\kappa^2),
\end{align}
where $p_r(\cdot)$ and $p_\ell(\cdot)$ are polynomials given by
\begin{align*}
\textstyle
    p_r(x) = c_m \prod_{i=1}^{m} (1-r_{1i}x)\quad \text{ and }\quad
    p_\ell(x) = b_{m+1}\prod_{j=1}^{m+1} (1-r_{2j}x), 
\end{align*}
and $\{r_{1i}\}_{i=1}^m$ and $\{r_{2j}\}_{j=1}^{m+1}$ are the roots of $q_1(x) =\sum_{i=0}^mc_ix^{i}$ and  $q_2(x)=\sum_{j=0}^{m+1}b_jx^{j}$, respectively. The coefficients  $\{c_i\}_{i=0}^m$ and  $\{b_j\}_{j=0}^{m+1}$ are determined as the best rational approximation $\sfrac{q_1}{q_2}$ of the function $x^{\beta-1}$ over the interval $J_h: = [\kappa^{2}\lambda_{N_h,h}^{-1}, \kappa^{2}\lambda_{1,h}^{-1}]$, where $\lambda_{1,h}, \lambda_{N_h,h}>0$ are the smallest and the largest eigenvalue of $L_h$, respectively. Interval $J_h$ contains the spectrum of $(L_h/\kappa^2)^{-1}$ and is a subset of $J:=[0,1]$, which in turn contains the spectrum of $(L/\kappa^2)^{-1}$.  We scale the operators $L$ and $L_h$ by dividing by $\kappa^2$ to avoid recomputing the rational approximation coefficients $\{c_i\}_{i=0}^m$ and $\{b_j\}_{j=0}^{m+1}$ for each mesh size $h$. This scaling ensures that the spectrum of $(L_h/\kappa^2)^{-1}$ lies within a fixed interval $J_* := [\delta, 1]$, where $\delta \in (0, 1)$ is chosen such that $J_h \subset J_*$ for all $h$. As a result, the rational approximation of $x^{\beta - 1}$ needs to be computed only once on $J_*$, regardless of the mesh size. For further details, see \citep[Section~3.5]{bolin2020rationalapprox}.

Replacing the rational approximation \eqref{rmdef} of $(L_h/\kappa^2)^{-\beta}$ in \eqref{weak_form_m2} yields 
\begin{align}
\label{weak_form_m21}
        (r_m((L_h/\kappa^2)^{-1})+\tau \kappa^{2\beta} I_{V_h})U_{h,m}^{k+1}=r_m((L_h/\kappa^2)^{-1})(U_{h,m}^{k}+\tau P_hf^{k+1}),
\end{align}
where the subindex $m$ indicates the dependence of the solution on the degree of the rational approximation. We are now in a position to analyze the impact of the rational approximation \eqref{rmdef} on the accuracy of the fully discrete scheme. The next result provides a complete error estimate for the scheme \eqref{weak_form_m21}, quantifying the combined effect of temporal discretization, spatial discretization, and rational approximation.
\begin{theorem} 
\label{lastlasttheorem}
Let $\alpha\in(0,2]$ and $u$ be the solution to \eqref{weakformofentireequation}. Assume further that for $k=0,\dots,N-1$, $U^{k+1}_{h,m}\in V_h$ solves the weak form of \eqref{weak_form_m21} with $U^{0}_{h,m}= U^{0}_{h}$. Suppose that $u_0\in\dot{H}^{\alpha}(\Gamma)$. Then, for every sufficiently small $\epsilon>0$,
\begin{align}
    \label{finalfinalest}
        \norm{u-U_{h,m}^\tau}_{L_2(0,T;L_2(\Gamma))}\lesssim\tau+h^{\alpha-\epsilon}K_\tau + \tau^{-2} h^{\alpha-2}\mathrm{e}^{-2\pi\sqrt{(1-\sfrac{\alpha}{2})m}},
    \end{align}
    where $K_\tau = |\ln(\tau)|^2$ if $f\in L_\infty(0,T;L_2(\Gamma))$ and $K_\tau = |\ln(\tau)|$ if $f\in L_\infty(0,T;\dot{H}^{\alpha}(\Gamma))$.
\end{theorem}

\begin{proof}[Proof of Theorem~\ref{lastlasttheorem}]
Let $\beta = \sfrac{\alpha}{2}$. By triangle inequality, 
    \begin{align*}
        \norm{u-U_{h,m}^\tau}_{L_2(0,T;L_2(\Gamma))} \leq \norm{u-U_h^\tau}_{L_2(0,T;L_2(\Gamma))} + \norm{U_h^\tau - U_{h,m}^\tau}_{L_2(0,T;L_2(\Gamma))}.
    \end{align*}
    Theorem \ref{theorem:theorem_about_error_bounded_by_tau_plus_h_alpha} provides a bound for the first term in the right hand side. Let $Z_h = \kappa^2L_h^{-1}$, $g(x) = x^{\beta}/(x^{\beta}+\tau\kappa^{2\beta})$, and $R_m(x) = r_m(x)/(r_m(x)+\tau\kappa^{2\beta} )$. By isolating $U_h^{k+1}$ in \eqref{weak_form_m2} and $U_{h,m}^{k+1}$ in \eqref{weak_form_m21}, and then iterating the resulting relations, we obtain
\begin{align*}
    U_h^{k+1}&=g^{k+1}(Z_h)U_h^{0}+\tau\textstyle\sum_{i=0}^kg^{k+1-i}(Z_h)P_hf^{i+1},\\
    U_{h,m}^{k+1}&= R^{k+1}_m(Z_h)U_{h}^{0}+\tau\textstyle\sum_{i=0}^kR_m^{k+1-i}(Z_h)P_hf^{i+1},
\end{align*}
respectively. Inserting these into $\|U_h^\tau-U_{h,m}^\tau\|^2_{\ell^2(L_2(\Gamma))} = \tau\sum_{k=0}^{N-1}\|U_h^{k+1}-U_{h,m}^{k+1}\|^2_{L_2(\Gamma)}$, and splitting the resulting expression, yields
\begin{align*}
        \norm{U_h^\tau-U_{h,m}^\tau}^2_{\ell^2(L_2(\Gamma))} &\leq 2\tau\textstyle\sum_{k=0}^{N-1}\norm{(g^{k+1}(Z_h)-R^{k+1}_m(Z_h))U_{h}^{0}}^2_{L_2(\Gamma)}\\
        &+2\tau\textstyle\sum_{k=0}^{N-1}\norm{\tau\sum_{i=0}^k(g^{k+1-i}(Z_h)-R_m^{k+1-i}(Z_h))P_hf^{i+1}}^2_{L_2(\Gamma)}\\
        &=:E^2_1+E^2_2.
\end{align*}
\paragraph{Estimate for $E^2_1$} By expanding $U_h^0$ in the eigenbasis $\{e_{j,h}\}_{j=1}^{N_h}$, we arrive at
\begin{align*}
        E^2_1&=2 \tau\textstyle\sum_{k=0}^{N-1}\sum_{j=1}^{N_h}\left((g^{k+1}(Z_h)-R^{k+1}_m(Z_h))U_{h}^{0}, e_{j,h}\right)^2_{L_2(\Gamma)}\\
        &= 2\tau\textstyle\sum_{k=0}^{N-1}\sum_{j=1}^{N_h}(g^{k+1}(\mu_j)-R_m^{k+1}(\mu_j))^2\left(P_hu_0, e_{j,h}\right)^2_{L_2(\Gamma)},
\end{align*}
where $\mu_j = \kappa^2\lambda_{j,h}^{-1}$. Since $0\leq g(x), R_m(x)\leq1$ for $x\in J_h$, we can use the inequality $|A^k-B^k|\leq k|A-B|$ with $A=g(x)$ and $B= R_m(x)$ to obtain 
\begin{align}
    E^2_1&\leq 2\tau\textstyle\sum_{k=0}^{N-1}(k+1)^2\sum_{j=1}^{N_h}(g(\mu_j)-R_m(\mu_j))^2\left(P_hu_0, e_{j,h}\right)^2_{L_2(\Gamma)}\notag\\
    &\leq 2\tau \norm{P_hu_0}^2_{L_2(\Gamma)}\textstyle\left(\sup_{x\in J_h}|g(x)-R_m(x)|\right)^2\sum_{k=0}^{N-1}(k+1)^2\notag\\
    &\leq 2\tau N^3 \norm{u_0}^2_{L_2(\Gamma)}\textstyle\left(\sup_{x\in J_h}|g(x)-R_m(x)|\right)^2. \label{vv2}
\end{align}
Since $|(x^\beta+\tau\kappa^{2\beta})(r_m(x)+\tau\kappa^{2\beta})|\geq \tau^2\kappa^{4\beta}$ for $x\in J_h$, we have that 
\begin{align}
\label{eq:bound_for_g_munis_Rm}
    |g(x)-R_m(x)| = \dfrac{\tau\kappa^{2\beta}|x^\beta-r_m(x)|}{|(x^\beta+\tau\kappa^{2\beta})(r_m(x)+\tau\kappa^{2\beta})|}\leq\dfrac{|x^\beta-r_m(x)|}{\tau\kappa^{2\beta}}.
\end{align}
Using the above, \eqref{vv2} becomes $E^2_1 \leq 2T^3 \norm{u_0}^2_{L_2(\Gamma)}\kappa^{-4\beta}\tau^{-4}\textstyle\left(\sup_{x\in J_h}|x^\beta-r_m(x)|\right)^2$. From the proof of \citep[Theorem~3.1]{bolin2020rationalapprox}, we have that
\begin{align}
\label{boundonsupforrational}
   \textstyle\left(\sup_{x\in J_h}|x^\beta-r_m(x)|\right)^2\lesssim  h^{-4(1-\beta)}\mathrm{e}^{-4\pi\sqrt{(1-\beta)m}}.
\end{align}
This and the previous estimate imply that 
\begin{align}
\label{ee1}
        E_1 &\lesssim \kappa^{-2\beta}\norm{u_0}_{L_2(\Gamma)}  \tau^{-2}h^{-2(1-\beta)}\mathrm{e}^{-2\pi\sqrt{(1-\beta)m}}.
\end{align}
\paragraph{Estimate for $E_2^2$} Expanding $P_hf^{i+1}$ in the eigenbasis $\{e_{j,h}\}_{j=1}^{N_h}$ yields
\begin{align*}
    E_2^2 &= 2\tau^3\textstyle\sum_{k=0}^{N-1}\sum_{j=1}^{N_h}\left(\sum_{i=0}^k(g^{k+1-i}(Z_h)-R_m^{k+1-i}(Z_h))P_hf^{i+1}, e_{j,h}\right)^2_{L_2(\Gamma)}\\
    &= 2\tau^3\textstyle\sum_{k=0}^{N-1}\sum_{j=1}^{N_h}\left(\sum_{i=0}^k (g^{k+1-i}(\mu_j)-R_m^{k+1-i}(\mu_j)) (P_hf^{i+1}, e_{j,h})_{L_2(\Gamma)}\right)^2.
\end{align*}
By Cauchy–Schwarz inequality, we now obtain that 
\begin{align*}
\textstyle
     E_2^2  \leq 2\tau^3\sum_{k=0}^{N-1}\sum_{j=1}^{N_h}\left(\sum_{i=0}^k(g^{k+1-i}(\mu_j)-R_m^{k+1-i}(\mu_j))^2\right)\left(\sum_{i=0}^k(f^{i+1}, e_{j,h})^2_{L_2(\Gamma)}\right).
\end{align*}
Combining \eqref{eq:bound_for_g_munis_Rm} with the estimate $\sum_{i=0}^k (k+1-i)^2 \le (k+1)^3 \le N^3$, the first inner sum can be bounded as follows.
\begin{align*}
\textstyle
\sum_{i=0}^k(g^{k+1-i}(\mu_j)-R_m^{k+1-i}(\mu_j))^2&\leq (g(\mu_j)-R_m(\mu_j))^2 \textstyle\sum_{i=0}^k(k+1-i)^2 \\
&\leq N^3\tau^{-2}\kappa^{-4\beta}\textstyle\left(\sup_{x\in J_h}|x^\beta-r_m(x)|\right)^2
\end{align*}
Therefore, the last expression for $E_2^2$ becomes
\begin{align*}
     E_2^2  &\leq 2\tau N^3\kappa^{-4\beta}\textstyle\left(\sup_{x\in J_h}|x^\beta-r_m(x)|\right)^2\sum_{k=0}^{N-1}\sum_{j=1}^{N_h}\sum_{i=0}^k(f^{i+1}, e_{j,h})^2_{L_2(\Gamma)}\\
     &= 2\tau N^3\kappa^{-4\beta}\textstyle\left(\sup_{x\in J_h}|x^\beta-r_m(x)|\right)^2\sum_{k=0}^{N-1}\sum_{i=0}^k\norm{f^{i+1}}^2_{L_2(\Gamma)}.
\end{align*}
The double summation in the above expression can be estimated as follows.
\begin{align*}
\textstyle
\sum_{k=0}^{N-1}\sum_{i=0}^k \norm{f^{i+1}}^2_{L_2(\Gamma)}& = \textstyle\sum_{i=0}^{N-1}(N-i) \norm{f^{i+1}}^2_{L_2(\Gamma)}\leq N \sum_{i=0}^{N-1} \norm{f^{i+1}}^2_{L_2(\Gamma)}\\
&= N\tau^{-1}\norm{f^\tau}^2_{\ell^2(L_2(\Gamma))}\lesssim N\tau^{-1}\norm{f}^2_{L_\infty(0,T;L_2(\Gamma))}.
\end{align*}
Using \eqref{boundonsupforrational} and replacing the above estimate into the last expression for $E^2_2$ gives
\begin{align*}
     E_2^2 & \leq 2\tau N^3\kappa^{-4\beta}\textstyle\left(\sup_{x\in J_h}|x^\beta-r_m(x)|\right)^2 N\tau^{-1}\norm{f}^2_{L_\infty(0,T;L_2(\Gamma))}\\
     &\lesssim 2T^4\kappa^{-4\beta}\norm{f}^2_{L_\infty(0,T;L_2(\Gamma))}\tau^{-4}h^{-4(1-\beta)}\mathrm{e}^{-4\pi\sqrt{(1-\beta)m}},
\end{align*}
which implies that
\begin{align}
\label{ee2}
     E_2\lesssim \kappa^{-2\beta}\norm{f}_{L_\infty(0,T;L_2(\Gamma))}\tau^{-2}h^{-2(1-\beta)}\mathrm{e}^{-2\pi\sqrt{(1-\beta)m}}.
\end{align}
By combining estimates \eqref{ee1} and \eqref{ee2}, we obtain that 
\begin{equation*}
    \norm{U_h^\tau-U_{h,m}^\tau}_{\ell^2(L_2(\Gamma))} \lesssim \kappa^{-2\beta}(\norm{u_0}_{L_2(\Gamma)} +\norm{f}_{L_\infty(0,T;L_2(\Gamma))}) \tau^{-2}h^{-2(1-\beta)}\mathrm{e}^{-2\pi\sqrt{(1-\beta)m}}.
\end{equation*}
By Theorem \ref{theorem:theorem_about_error_bounded_by_tau_plus_h_alpha} and the above estimate, we obtain \eqref{finalfinalest}.
\end{proof}

Finally, to derive the matrix representation of the numerical scheme, which is what we use to implement it, we first use \eqref{rmdef} to rewrite \eqref{weak_form_m21} as
\begin{align}
\label{weak_form_m3}
        (p_r(L_h/\kappa^2)+\tau \kappa^{2\beta}p_\ell(L_h/\kappa^2))U_{h,m}^{k+1}= p_r(L_h/\kappa^2) (U_{h,m}^{k}+\tau f^{k+1}).
\end{align}
We can now go one step further and isolate the numerical solution at time $t_{k+1}$ by considering the following partial fraction decomposition 
\begin{align}
\label{eq:partial_fraction}
    p_r(x)(p_r(x)+\tau \kappa^{2\beta}p_\ell(x))^{-1} = \textstyle \sum_{n=1}^{m+1}a_n(x-z_n)^{-1} + r,
\end{align}
where $\{a_n\}_{n=1}^{m+1}$ and $\{z_n\}_{n=1}^{m+1}$ are the residues and poles, respectively, and $r$ denotes the remainder. Substituting this into \eqref{weak_form_m3} gives the iteration rule
\begin{align}
\label{weak_form_m4}
\textstyle
        U_{h,m}^{k+1}= \left(\sum_{n=1}^{m+1}a_n(L_h/\kappa^2-z_nI_{V_h})^{-1} + rI_{V_h}\right) (U_{h,m}^{k}+\tau f^{k+1}).
\end{align}
At each time step $t_k$, the finite element solution $U_{h,m}^k\in V_h$ to the weak form of \eqref{weak_form_m21} can be expressed as a linear combination of the basis functions  $\{\psi^i_h\}_{i=1}^{N_h}$ introduced in Section~\ref{sec:finite_element_approximation}, namely, $U_{h,m}^k(s) =  \sum_{i=1}^{N_h}u_i^k\psi^i_h(s)$, $s\in\Gamma$. Substituting this expansion into the weak form of \eqref{weak_form_m4} yields the following matrix iteration
\begin{align}
\label{eq:final_scheme2}
\textstyle
\mathbf{U}_{k+1} = \mathbf{C}^{-1}\left(\sum_{n=1}^{m+1} a_n\left(\mathbf{L}\mathbf{C}^{-1}/\kappa^2-z_n\mathbf{I}_{N_h}\right)^{-1} + r\mathbf{I}_{N_h}\right) (\mathbf{C}\mathbf{U}_k+\tau \mathbf{F}_{k+1}),
\end{align}
where $\mathbf{L} = \kappa^2\mathbf{C}+\mathbf{G}$, the matrix $\mathbf{C}\in\mathbb{R}^{N_h\times N_h}$ has entries $\mathbf{C}_{i,j} = (\psi^i_h,\psi^j_h)_{L_2(\Gamma)}$, $\mathbf{G}\in\mathbb{R}^{N_h\times N_h}$ has entries $\mathbf{G}_{i,j} = ({\psi^{i}_h}',{\psi^{j}_h}')_{L_2(\Gamma)}$, $\mathbf{U}_k\in\mathbb{R}^{N_h}$ has entries $u_i^k$, and $\mathbf{F}_k\in\mathbb{R}^{N_h}$ has components $(f^{k},\psi^i_h)_{L_2(\Gamma)}$. In practice, since the rational function in \eqref{eq:partial_fraction} is proper, there is no remainder $r$. Moreover, since $( \mathbf{L}\mathbf{C}^{-1}/\kappa^2-z_n\mathbf{I}_{N_h})^{-1}  = \mathbf{C}(\mathbf{L}/\kappa^2-z_n\mathbf{C})^{-1}$, scheme \eqref{eq:final_scheme2} reduces to
\begin{align}
\label{thenumericalscheme}
\textstyle
\mathbf{U}_{k+1} = \left(\sum_{n=1}^{m+1} a_n\left(\mathbf{L}/\kappa^2-z_n\mathbf{C}\right)^{-1}\right) (\mathbf{C}\mathbf{U}_k+\tau \mathbf{F}_{k+1}),
\end{align}
where it is evident that only sparse solves are required.
\section{Numerical experiments} 
This section presents numerical experiments to validate the theoretical error estimates established in the previous sections.
\label{numericalimplementationsec}
\subsection{Eigenfunction-based construction of an exact solution}
To derive an exact solution to problem \eqref{eq:maineq} for a numerical experiment, we consider the tadpole graph depicted in Figure~\ref{basis}. This graph is selected because the eigenvalues and eigenfunctions of its associated shifted Kirchhoff--Laplacian are explicitly known (see \citep[Example~1]{bolin2024gaussian} for details), which enables the construction of an exact solution via the representation~\eqref{eq:sol_reprentation}. To achieve this, we expand both the initial condition and the right-hand side in terms of the eigenfunctions. For the initial condition $u_0$, we select coefficients $\{x_j\}_{j=1}^{\infty}$ and set $u_0(s) = \sum_{j=1}^{\infty}x_je_j(s)$. Similarly, for the right-hand side function $f$, we choose a scalar function $g(t)$ and coefficients $\{y_j\}_{j=1}^{\infty}$ and let $f(s,t) = g(t)\sum_{j=1}^{\infty} y_j e_j(s)$. This approach has the additional advantage that the integrals $\left(u_0, e_j\right)_{L_2(\Gamma)}$ and $\left(f(r), e_j\right)_{L_2(\Gamma)}$ appearing in~\eqref{eq:sol_reprentation} do not require numerical approximation. Substituting these expressions into the solution formula~\eqref{eq:sol_reprentation} yield
\begin{align}
\label{sollll}
    u(s,t) = \textstyle\sum_{j=1}^{\infty}(x_j+y_j G_j(t))\mathrm{e}^{-\lambda^{\sfrac{\alpha}{2}}_jt}e_j(s),\quad G_j(t)= \int_0^t \mathrm{e}^{\lambda^{\sfrac{\alpha}{2}}_jr}g(r)\dd r.
\end{align}
For suitable choices of $g(r)$, the integral defining $G_j(t)$ in \eqref{sollll} can be evaluated in closed form. Although the series in~\eqref{sollll} is written as an infinite sum, in our construction only finitely many coefficients $\{x_j\}_{j=1}^{N_0}$ and $\{y_j\}_{j=1}^{N_f}$ are nonzero. Consequently, the solution is given exactly by a finite sum, with no need for series truncation or numerical integration. 
\subsection{Numerical assembly and approximation of the error}
\label{errorcompsec}
For the numerical implementation, however, the integrals $(f^{k},\psi^i_h)_{L_2(\Gamma)}$, which determine the components of the vector $\mathbf{F}_k$ at each time step $t_k$, must be approximated numerically, as exact evaluation of these integrals is generally not feasible. Let $\mathbf{F}\in\mathbb{R}^{N_h\times (N+1)}$ denote the matrix whose columns are $\mathbf{F}_k$. This matrix has entries $\mathbf{F}_{i,k}$ given by
\begin{align*}
   \textstyle (f^{k},\psi^i_h)_{L_2(\Gamma)} = (f(\cdot,t_k),\psi^i_h)_{L_2(\Gamma)} = \sum_{j=1}^{N_f} y_j g(t_k) (e_j,\psi^i_h)_{L_2(\Gamma)}.
\end{align*}
To approximate the inner products $(e_j,\psi^i_h)_{L_2(\Gamma)}$, we use the quadrature formula $\boldsymbol{\psi}_i^\top\mathbf{C}^{\star}\mathbf{e}_j$, where $\mathbf{e}_j=[e_j(s_1)\dots e_j(s_{N_{h_{\star}}})]^\top\in\mathbb{R}^{N_{h_\star}}$ and $\boldsymbol{\psi}_i=[\psi^i_h(s_1)\dots \psi^i_h(s_{N_{h^\star}})]^\top\in\mathbb{R}^{N_{h_\star}}$  are vectors of function evaluations on a fine spatial mesh with mesh size $h_{\star}$ and nodes $\{s_\ell\}_{\ell=1}^{N_{h_{\star}}}$. Matrix $\mathbf{C}^{\star}\in\mathbb{R}^{N_{h_\star}\times N_{h_\star}}$ contains the corresponding quadrature weights, with entries $\mathbf{C}^{\star}_{i,j} = (\psi_{h_\star}^i,\psi_{h_\star}^j)_{L_2(\Gamma)}$ for $i,j = 1,\dots, N_{h_\star}$. We emphasize that two spatial meshes are involved in this construction. The basis functions $\{\psi^i_h\}_{i=1}^{N_h}$ are defined on a coarse spatial mesh with mesh size $h$, while the quadrature is carried out over a finer mesh with associated basis functions $\{\psi^\ell_{h_\star}\}_{\ell=1}^{N_{h_\star}}$ and mesh size $h_\star$.
If $\mathbf{N}\in\mathbb{R}^{N_{f}\times (N+1)}$ denotes the matrix with entries $\mathbf{N}_{j,k} = y_j g(t_k)$, then $\mathbf{F}$ can be approximated as $[\boldsymbol{\psi}_1\dots \boldsymbol{\psi}_{N_h}]^\top \mathbf{C}^{\star} [\mathbf{e}_1\dots \mathbf{e}_{N_f}] \mathbf{N}$.

Let $E$ denote a quadrature approximation of the $L_2(0,T;L_2(\Gamma))$-error on the left-hand side of \eqref{finalfinalest}. That is, $E^2 = \mathbf{w}^\top_{\star}(\boldsymbol{\mathfrak{U}}- \mathbf{\Psi}p(\mathbf{U}))^2\boldsymbol{\tau}_{\star}$, where $\mathbf{X}^2$ means entry-wise square. Here, $\boldsymbol{\mathfrak{U}}\in\mathbb{R}^{N_{h_\star}\times (N_\star+1)}$ is a matrix containing evaluations of the exact solution, with entries $\boldsymbol{\mathfrak{U}}_{i,k} = u(s_i,t_k)$ where $s_i$ is the $i$th location in a fine spatial mesh with mesh size $h_{\star}$, and $t_k$ is the $k$th time step in a fine temporal mesh with step size $\tau_{\star} = T/N_\star$. The matrix $\mathbf{U}\in\mathbb{R}^{N_{h}\times (N+1)}$ has the numerical solution, with columns $\mathbf{U}_k$ computed by \eqref{thenumericalscheme}. This solution is defined on possibly coarser spatial and temporal meshes (with mesh sizes $h$ and $\tau=T/N$, respectively) and is appropriately projected onto the fine temporal mesh as a piecewise constant function $p(\mathbf{U})\in\mathbb{R}^{N_{h}\times (N_\star+1)}$.  The matrix $\mathbf{\Psi}\in\mathbb{R}^{N_{h_\star}\times N_h}$, with entries $\mathbf{\Psi}_{i,j}=\psi^j_h(s_i)$ for $j=1\dots, N_h$ and $i = 1,\dots N_{h_{\star}}$, projects the numerical solution onto the fine spatial mesh. If the fine and coarse meshes coincide (i.e., $N_h = N_{h_{\star}}$), then $\mathbf{\Psi} = \mathbf{I}_{N_{h_{\star}}}$. The vector $\mathbf{w}_{\star}\in\mathbb{R}^{N_{h_\star}}$ contains weights $w_i^\star = (\psi^i_{h_{\star}},1)_{L_2(\Gamma)}$ and $\boldsymbol{\tau}_{\star}\in\mathbb{R}^{N_\star+1}$ contains uniform weights equal to $\tau_{\star}$.

\subsection{Convergence behavior}
The total error bound \eqref{finalfinalest} in Theorem \ref{lastlasttheorem} has three components: the temporal error $\mathcal{O}(\tau)$, the spatial error $\mathcal{O}(h^{\alpha-\epsilon}|\ln(\tau)|)$ (since the chosen $f$ has $ L_\infty(0,T;\dot{H}^{\alpha}(\Gamma))$ regularity), and the rational approximation error $\mathcal{O}(\tau^{-2}h^{\alpha-2}\text{e}^{-2\pi\sqrt{(1-\sfrac{\alpha}{2})m}})$. We next outline a systematic approach to verify the convergence behavior with respect to each of the parameters $h$, $\tau$, and $m$, by appropriately balancing the remaining error components. 

\paragraph{Convergence in $h$}
\label{convh}
To observe convergence in $h$, we determine $\tau$ numerically from the relation $\tau\propto h^{\alpha-\epsilon}|\ln(\tau)|$ and choose $m\propto\left\lceil\ln^2\!\bigl(h^{2-\epsilon}|\ln(\tau)|\tau^2\bigr)\big/\!\left(4\pi^2\left(1-\sfrac{\alpha}{2}\right)\right)\right\rceil$ so that the temporal and rational approximation errors are of the same order as the spatial error. This yields $E_h=E/|\ln(\tau)|\lesssim h^{\alpha-\epsilon}$, corresponding to a convergence rate of order $\alpha-\epsilon$ with respect to the mesh size $h$. Fitting the regression $\log_{10}E_h=s_h\log_{10}h+\log_{10}c_h$, we therefore expect $s_h\approx\alpha-\epsilon$.

\paragraph{Convergence in $\tau$}
\label{convtau}
To observe convergence in $\tau$, we choose $h \propto (\tau/|\ln(\tau)|)^{\sfrac{1}{(\alpha-\epsilon)}}$ and $m\propto\left\lceil\ln^2\!\bigl(h^{2-\epsilon}|\ln(\tau)|\tau^{2}\bigr)\big/\!\left(4\pi^2\left(1-\sfrac{\alpha}{2}\right)\right)\right\rceil$. This yields $E_\tau = E\lesssim\tau$, corresponding to a convergence rate of order $1$ with respect to the time step $\tau$. Fitting the regression $\log_{10}E_\tau=s_\tau\log_{10}\tau+\log_{10}c_\tau$, we therefore expect $s_\tau\approx1$.

\paragraph{Convergence in $m$} 
\label{convm}
To observe convergence in $m$, we determine $h$ numerically from the relation $|\ln((C_mh^{\alpha-2})^{\sfrac{1}{3}})| \propto C_m^{\sfrac{1}{3}}h^{\epsilon-\sfrac{(2\alpha+2)}{3}}$ and choose $\tau \propto (C_mh^{\alpha-2})^{\sfrac{1}{3}}$, where $C_m = \mathrm{e}^{-2\pi\sqrt{(1-\sfrac{\alpha}{2})m}}$. This yields $E_m = E/(\tau^{-2}h^{\alpha-2})\lesssim C_m$, corresponding to exponential decay with respect to the rational approximation order $m$. Fitting the regression $\ln E_m = s_m\sqrt{m} +\ln c_m$, we therefore expect $s_m\approx -2\pi\sqrt{1-\sfrac{\alpha}{2}}$.

\subsection{Results}

To verify the convergence behavior with respect to $h$, $\tau$, and $m$, we fix $T = 2$, $\epsilon =  10^{-10}$, and $\kappa = 4$. The initial condition $u_0$ is defined by choosing the coefficient $x_5=10$ and setting all other $x_j$ to zero, which gives $u_0(s) = 10e_5(s)$. Similarly, the right-hand side function $f$ is constructed by choosing $y_7=10$ and all other $y_j=0$, and by setting $g(t)=\sin(\pi t)$. This yields $f(s,t) = 10\sin(\pi t)e_7(s)$, and the corresponding exact solution takes the form $u(s,t) = x_5\mathrm{e}^{-\lambda^{\sfrac{\alpha}{2}}_5t}e_5(s)+y_7 G_7(t)\mathrm{e}^{-\lambda^{\sfrac{\alpha}{2}}_7t}e_7(s)$, where $G_7(t)= \int_0^t \mathrm{e}^{\lambda^{\sfrac{\alpha}{2}}_7r}\sin(\pi r)\dd r$. We evaluate the exact solution on a fine space-time mesh with spatial and temporal resolutions given by $h_{\star} = 9\cdot10^{-4}$ and $\tau_{\star} = 5\cdot10^{-6}$, respectively. The same fine spatial mesh is also used to approximate the inner products $(e_7, \psi^i_h)_{L_2(\Gamma)}$. To study convergence with respect to $h$, we consider six logarithmically spaced mesh sizes ranging from $9\cdot10^{-4}$ to $10^{-2}$, inclusive, balancing $\tau$ and $m$ accordingly. For each value of $h$, the corresponding error $E_h$ is computed as described in Section~\ref{errorcompsec}. A linear regression is then performed on the pairs $(\log_{10} h, \log_{10} E_h)$ to estimate the convergence rate $\alpha-\epsilon$, following the procedure outlined in Section~\ref{convh}. We repeat this routine for $\alpha = 1.0,1.2,1.4,1.6,1.8$. The results are shown in the left panel of Figure~\ref{fig:combined_cov_rates}. Similarly, to investigate convergence with respect to $\tau$, we consider six logarithmically spaced time steps ranging from $9\cdot10^{-4}$ to $10^{-2}$, inclusive, balancing $h$ and $m$ accordingly. For each value of $\tau$, the corresponding error $E_\tau$ is computed as described in Section~\ref{errorcompsec}. A linear regression is then performed on the pairs $(\log_{10} \tau, \log_{10} E_\tau)$ to estimate the convergence rate $1$, following the procedure outlined in Section~\ref{convtau}. The results are displayed in the center panel of Figure~\ref{fig:combined_cov_rates}. Finally, to examine convergence with respect to $m$, we vary $m$ over the values $1, 2, 3, 4$, balancing $h$ and $\tau$ accordingly. For each value of $m$, the corresponding error $E_m$ is computed as described in Section~\ref{errorcompsec}. A linear regression is then performed on the pairs $(\sqrt{m}, \ln E_m)$ to estimate the exponential decay $-2\pi\sqrt{1-\sfrac{\alpha}{2}}$, following the approach outlined in Section~\ref{convm}. The results are shown in the right panel of Figure~\ref{fig:combined_cov_rates}. Together, the convergence trends observed across all panels of Figure~\ref{fig:combined_cov_rates} consistently align with the theoretical predictions, as reflected by the relative percentage errors reported in the legends, the largest of which is only 2.86\%. This confirms the accuracy and reliability of the numerical scheme and the underlying error analysis.

\begin{figure}[t]
\centering
\includegraphics[width=0.99\textwidth]{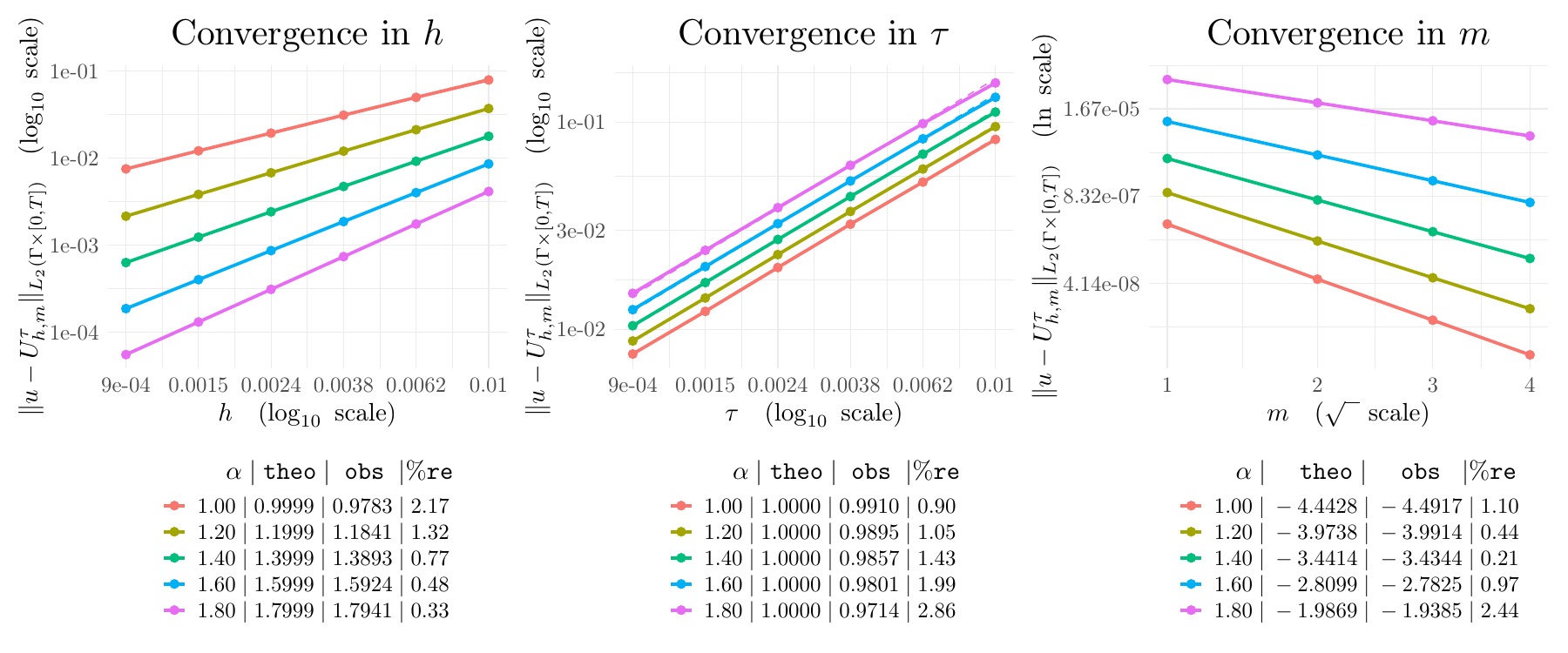}
\caption{Comparison between the theoretical and observed convergence behavior of the $\|u-U_{h,m}^\tau\|_{L_2(0,T;L_2(\Gamma))}$-error with respect to $h$, $\tau$, and $m$. The left and center panels display the convergence with respect to $h$ and $\tau$, respectively, on a $\mathrm{log}_{10}$-$\mathrm{log}_{10}$ scale, while the right plot shows the exponential decay in $m$ on a semi-$\ln$ scale, with $m$ plotted on a square-root scale. Dashed lines indicate the theoretical rates, and solid lines represent the observed error curves. 
The legends report the value of $\alpha$, the theoretical $(\mathtt{theo})$ and observed $(\mathtt{obs})$ values, together with the relative percentage error $(\mathtt{\%re})=100|\mathtt{theo}-\mathtt{obs}|/|\mathtt{theo}|$. For the $h$-, $\tau$-, and $m$-studies, $(\mathtt{theo},\mathtt{obs})$ corresponds to $(\alpha-\epsilon,s_h)$, $(1,s_\tau)$, and $(-2\pi\sqrt{1-\sfrac{\alpha}{2}},s_m)$, respectively.}
\label{fig:combined_cov_rates}
\end{figure}

\section{Conclusions and Future Work}
\label{conclusion}
We have presented a comprehensive numerical framework for the approximation of fractional diffusion equations on compact metric graphs. By combining recent analytical developments for fractional operators on metric graphs with efficient numerical techniques, it establishes a foundation for the study and computation of nonlocal diffusion processes on network domains.

There are several natural extensions of this work for future investigation. First, the framework can be generalized to incorporate operators $L =\kappa^2 -\nabla\cdot(\mathbf{H}\nabla)$, where $\mathbf{H}$ is a symmetric and positive definite matrix. Second, our approach can be extended to fractional time derivatives, replacing the classical time derivative with a Caputo or Riemann–Liouville derivative. Inclusion of memory effects would improve modeling capabilities for anomalous diffusion processes \citep{Evangelista2018Fractional, Metzler2000Therandomwalksguide}. Third, building on  \citep{Glusa2021errorestimates}, a natural extension is the study of optimal control problems governed by space-time fractional models on metric graphs. Fourth, although we focused on metric graphs, the proposed spectral-rational framework relies only on the functional calculus of self-adjoint operators with compact resolvent. Consequently, it is expected to extend naturally to fractional diffusion problems on bounded Euclidean domains and compact Riemannian manifolds using standard finite element discretizations. Finally, replacing the right-hand side with a stochastic forcing term, particularly space-time white noise, would facilitate modeling stochastic fractional diffusion processes on graphs.

\bibliographystyle{chicago}
\bibliography{references}

@Inbook{Henry1981,
author="Henry, Daniel",
title="Preliminaries",
bookTitle="Geometric Theory of Semilinear Parabolic Equations",
year="1981",
publisher="Springer Berlin Heidelberg",
address="Berlin, Heidelberg",
pages="3--40",
isbn="978-3-540-38528-8",
      NOTE = "\url{https://doi.org/10.1007/BFb0089649}",
}

@article {bolin2020numerical,
    AUTHOR = {Bolin, David and Kirchner, Kristin and Kov\'acs, Mih\'aly},
     TITLE = {Numerical solution of fractional elliptic stochastic {PDE}s
              with spatial white noise},
   JOURNAL = {IMA J. Numer. Anal.},
  FJOURNAL = {IMA Journal of Numerical Analysis},
    VOLUME = {40},
      YEAR = {2020},
    NUMBER = {2},
     PAGES = {1051--1073},
      ISSN = {0272-4979,1464-3642},
   MRCLASS = {65N75 (35J25 35R11 35R60 60H15 60H35 65N30)},
  MRNUMBER = {4092278},
      NOTE = {\url{https://doi.org/10.1093/imanum/dry091}},
}

@article {ChandlerWilde2015interpolation,
    AUTHOR = {Chandler-Wilde, S. N. and Hewett, D. P. and Moiola, A.},
     TITLE = {Interpolation of {H}ilbert and {S}obolev spaces: quantitative
              estimates and counterexamples},
   JOURNAL = {Mathematika},
  FJOURNAL = {Mathematika. A Journal of Pure and Applied Mathematics},
    VOLUME = {61},
      YEAR = {2015},
    NUMBER = {2},
     PAGES = {414--443},
      ISSN = {0025-5793,2041-7942},
   MRCLASS = {46B70 (46E35)},
  MRNUMBER = {3343061},
MRREVIEWER = {Oscar\ Dom\'inguez},
      NOTE = {\url{https://doi.org/10.1112/S0025579314000278}},
}

@article {bolin2023regularity,
    AUTHOR = {Bolin, David and Kov\'acs, Mih\'aly and Kumar, Vivek and
              Simas, Alexandre B.},
     TITLE = {Regularity and numerical approximation of fractional elliptic
              differential equations on compact metric graphs},
   JOURNAL = {Math. Comp.},
  FJOURNAL = {Mathematics of Computation},
    VOLUME = {93},
      YEAR = {2024},
    NUMBER = {349},
     PAGES = {2439--2472},
      ISSN = {0025-5718,1088-6842},
   MRCLASS = {65M12 (35A01 35A02 35R02 60H15 60H40)},
  MRNUMBER = {4759380},
      NOTE = {\url{https://doi.org/10.1090/mcom/3929}},
}

@article {arioli2018finite,
    AUTHOR = {Arioli, Mario and Benzi, Michele},
     TITLE = {A finite element method for quantum graphs},
   JOURNAL = {IMA J. Numer. Anal.},
  FJOURNAL = {IMA Journal of Numerical Analysis},
    VOLUME = {38},
      YEAR = {2018},
    NUMBER = {3},
     PAGES = {1119--1163},
      ISSN = {0272-4979,1464-3642},
   MRCLASS = {65N30 (05C10 35R02 81Q35)},
  MRNUMBER = {3829156},
      NOTE = {\url{https://doi.org/10.1093/imanum/drx029}},
       URL = {https://doi.org/10.1093/imanum/drx029},
}

@article {bolin2020rationalapprox,
    AUTHOR = {Bolin, David and Kirchner, Kristin},
     TITLE = {The rational {SPDE} approach for {G}aussian random fields with
              general smoothness},
   JOURNAL = {J. Comput. Graph. Statist.},
  FJOURNAL = {Journal of Computational and Graphical Statistics},
    VOLUME = {29},
      YEAR = {2020},
    NUMBER = {2},
     PAGES = {274--285},
      ISSN = {1061-8600,1537-2715},
   MRCLASS = {62M30 (60G60 60H15 60H35 62-08)},
  MRNUMBER = {4116041},
      NOTE = {\url{https://doi.org/10.1080/10618600.2019.1665537}},
}

@phdthesis{kups73182,
            year = {2024},
          author = {Anna Weller},
           title = {Numerical Methods for Parabolic Partial Differential Equations on Metric Graphs},
          school = {Universit{\"a}t zu K{\"o}ln},
      NOTE = {\url{https://kups.ub.uni-koeln.de/73182/}}
}

@article {bolin2024gaussian,
    AUTHOR = {Bolin, David and Simas, Alexandre B. and Wallin, Jonas},
     TITLE = {Gaussian {W}hittle-{M}at\'ern fields on metric graphs},
   JOURNAL = {Bernoulli},
  FJOURNAL = {Bernoulli. Official Journal of the Bernoulli Society for
              Mathematical Statistics and Probability},
    VOLUME = {30},
      YEAR = {2024},
    NUMBER = {2},
     PAGES = {1611--1639},
      ISSN = {1350-7265,1573-9759},
   MRCLASS = {60G60 (35J25 35R60 60G15 60G17 60H15)},
  MRNUMBER = {4699566},
      NOTE = {\url{https://doi.org/10.3150/23-bej1647}},
}

@book {berkolaiko2013introduction,
    AUTHOR = {Berkolaiko, Gregory and Kuchment, Peter},
     TITLE = {Introduction to quantum graphs},
    SERIES = {Mathematical Surveys and Monographs},
    VOLUME = {186},
 PUBLISHER = {American Mathematical Society, Providence, RI},
      YEAR = {2013},
     PAGES = {xiv+270},
      ISBN = {978-0-8218-9211-4},
   MRCLASS = {81Q35 (05C90 31C20 34B24 34B45 81Q50)},
  MRNUMBER = {3013208},
MRREVIEWER = {Delio\ Mugnolo},
      NOTE = {\url{https://doi.org/10.1090/surv/186}},
}

@article {Glusa2021errorestimates,
    AUTHOR = {Glusa, Christian and Ot\'arola, Enrique},
     TITLE = {Error estimates for the optimal control of a parabolic
              fractional {PDE}},
   JOURNAL = {SIAM J. Numer. Anal.},
  FJOURNAL = {SIAM Journal on Numerical Analysis},
    VOLUME = {59},
      YEAR = {2021},
    NUMBER = {2},
     PAGES = {1140--1165},
      ISSN = {0036-1429,1095-7170},
   MRCLASS = {49M41 (35R11 49J20 49K20 49M25 65M12 65M15 65M60)},
  MRNUMBER = {4250571},
MRREVIEWER = {Lino\ J.\ \'Alvarez-V\'azquez},
      NOTE = {\url{https://doi.org/10.1137/19M1267581}},
}

@article {bottcher2024dynamicalprocessesmetricnetworks,
    AUTHOR = {B\"ottcher, Lucas and Porter, Mason A.},
     TITLE = {Dynamical processes on metric networks},
   JOURNAL = {SIAM J. Appl. Dyn. Syst.},
  FJOURNAL = {SIAM Journal on Applied Dynamical Systems},
    VOLUME = {24},
      YEAR = {2025},
    NUMBER = {4},
     PAGES = {2848--2885},
      ISSN = {1536-0040},
   MRCLASS = {35R02 (05C82 37E99 81Q35)},
  MRNUMBER = {4973741},
      NOTE = {\url{https://doi.org/10.1137/24M1628153}},
       URL = {https://doi.org/10.1137/24M1628153},
}

@Manual{MetricGraphRpackage,
    title = {MetricGraph: Random fields on metric graphs},
    author = {David Bolin and Alexandre B. Simas and Jonas Wallin},
    year = {2023},
      NOTE = {R package version 1.3.0.9000. \url{https://CRAN.R-project.org/package=MetricGraph}},
  }

@Manual{rSPDERpackage,
    title = {rSPDE: Rational Approximations of Fractional Stochastic Partial Differential Equations},
    author = {David Bolin and Alexandre B. Simas},
    year = {2023},
      NOTE = {R package version 2.3.3. \url{https://CRAN.R-project.org/package=rSPDE}},
  }

@Manual{Rsoftware,
    title = {R: A Language and Environment for Statistical Computing},
    author = {{R Core Team}},
    organization = {R Foundation for Statistical Computing},
    address = {Vienna, Austria},
    year = {2023},
      NOTE = {\url{https://www.R-project.org/}},
  }

@incollection {Ainsworth2018Towards,
    AUTHOR = {Ainsworth, Mark and Glusa, Christian},
     TITLE = {Towards an efficient finite element method for the integral
              fractional {L}aplacian on polygonal domains},
 BOOKTITLE = {Contemporary computational mathematics---a celebration of the
              80th birthday of {I}an {S}loan. {V}ol. 1, 2},
     PAGES = {17--57},
 PUBLISHER = {Springer, Cham},
      YEAR = {2018},
      ISBN = {978-3-319-72455-3; 978-3-319-72456-0},
   MRCLASS = {65N30 (35R11)},
  MRNUMBER = {3822227},
      NOTE = {\url{https://doi.org/10.1007/978-3-319-72456-0_2}},
}

@article {Caffarelli2007Extension,
    AUTHOR = {Caffarelli, Luis and Silvestre, Luis},
     TITLE = {An extension problem related to the fractional {L}aplacian},
   JOURNAL = {Comm. Partial Differential Equations},
  FJOURNAL = {Communications in Partial Differential Equations},
    VOLUME = {32},
      YEAR = {2007},
    NUMBER = {7-9},
     PAGES = {1245--1260},
      ISSN = {0360-5302,1532-4133},
   MRCLASS = {35J70},
  MRNUMBER = {2354493},
MRREVIEWER = {Francesco\ Petitta},
      NOTE = {\url{https://doi.org/10.1080/03605300600987306}},
}

@article {Bonito2018Numericalmethods,
    AUTHOR = {Bonito, Andrea and Borthagaray, Juan Pablo and Nochetto,
              Ricardo H. and Ot\'arola, Enrique and Salgado, Abner J.},
     TITLE = {Numerical methods for fractional diffusion},
   JOURNAL = {Comput. Vis. Sci.},
  FJOURNAL = {Computing and Visualization in Science},
    VOLUME = {19},
      YEAR = {2018},
    NUMBER = {5-6},
     PAGES = {19--46},
      ISSN = {1432-9360,1433-0369},
   MRCLASS = {65N30 (26A33 65N15)},
  MRNUMBER = {3893441},
      NOTE = {\url{https://doi.org/10.1007/s00791-018-0289-y}},
}

@article {Nochetto2016APDEapproach,
    AUTHOR = {Nochetto, Ricardo H. and Ot\'arola, Enrique and Salgado, Abner
              J.},
     TITLE = {A {PDE} approach to space-time fractional parabolic problems},
   JOURNAL = {SIAM J. Numer. Anal.},
  FJOURNAL = {SIAM Journal on Numerical Analysis},
    VOLUME = {54},
      YEAR = {2016},
    NUMBER = {2},
     PAGES = {848--873},
      ISSN = {0036-1429,1095-7170},
   MRCLASS = {65M60 (26A33 35R11 65M12 65M15)},
  MRNUMBER = {3478958},
MRREVIEWER = {Mohammad\ Asadzadeh},
      NOTE = {\url{https://doi.org/10.1137/14096308X}},
}

@article {Antil2016Aspacetime,
    AUTHOR = {Antil, Harbir and Ot\'arola, Enrique and Salgado, Abner J.},
     TITLE = {A space-time fractional optimal control problem: analysis and
              discretization},
   JOURNAL = {SIAM J. Control Optim.},
  FJOURNAL = {SIAM Journal on Control and Optimization},
    VOLUME = {54},
      YEAR = {2016},
    NUMBER = {3},
     PAGES = {1295--1328},
      ISSN = {0363-0129,1095-7138},
   MRCLASS = {49M25 (26A33 35J70 49J20 65M12 65M15 65M60 65R10)},
  MRNUMBER = {3504977},
      NOTE = {\url{https://doi.org/10.1137/15M1014991}},
}

@article {Cordoni2017Stochasticreactiondiffusion,
    AUTHOR = {Cordoni, Francesco and Di Persio, Luca},
     TITLE = {Stochastic reaction-diffusion equations on networks with
              dynamic time-delayed boundary conditions},
   JOURNAL = {J. Math. Anal. Appl.},
  FJOURNAL = {Journal of Mathematical Analysis and Applications},
    VOLUME = {451},
      YEAR = {2017},
    NUMBER = {1},
     PAGES = {583--603},
      ISSN = {0022-247X,1096-0813},
   MRCLASS = {35R02 (35K57 35R60)},
  MRNUMBER = {3619253},
      NOTE = {\url{https://doi.org/10.1016/j.jmaa.2017.02.008}},
}

@article {Pastor2015Epidemicprocess,
    AUTHOR = {Pastor-Satorras, Romualdo and Castellano, Claudio and Van
              Mieghem, Piet and Vespignani, Alessandro},
     TITLE = {Epidemic processes in complex networks},
   JOURNAL = {Rev. Modern Phys.},
  FJOURNAL = {Reviews of Modern Physics},
    VOLUME = {87},
      YEAR = {2015},
    NUMBER = {3},
     PAGES = {925--979},
      ISSN = {0034-6861,1539-0756},
   MRCLASS = {94C99},
  MRNUMBER = {3406040},
      NOTE = {\url{https://doi.org/10.1103/RevModPhys.87.925}},
}

@article {Herty2010Anewmodelforgasflow,
    AUTHOR = {Herty, M. and Mohring, J. and Sachers, V.},
     TITLE = {A new model for gas flow in pipe networks},
   JOURNAL = {Math. Methods Appl. Sci.},
  FJOURNAL = {Mathematical Methods in the Applied Sciences},
    VOLUME = {33},
      YEAR = {2010},
    NUMBER = {7},
     PAGES = {845--855},
      ISSN = {0170-4214,1099-1476},
   MRCLASS = {76N15},
  MRNUMBER = {2662309},
      NOTE = {\url{https://doi.org/10.1002/mma.1197}},
}

@article {Du2018PDE,
    AUTHOR = {Du, Bo and Lian, Xiuguo and Cheng, Xiwang},
     TITLE = {Partial differential equation modeling with {D}irichlet
              boundary conditions on social networks},
   JOURNAL = {Bound. Value Probl.},
  FJOURNAL = {Boundary Value Problems},
      YEAR = {2018},
     PAGES = {Paper No. 50, 11},
      ISSN = {1687-2762,1687-2770},
   MRCLASS = {91D30 (35K20 35K57 35K58)},
  MRNUMBER = {3797149},
      NOTE = {\url{https://doi.org/10.1186/s13661-018-0964-4}},
}

@book {Barrat2008Dynamicalprocess,
    AUTHOR = {Barrat, Alain and {Barth\'elemy}, Marc and Vespignani,
              Alessandro},
     TITLE = {Dynamical processes on complex networks},
 PUBLISHER = {Cambridge University Press, Cambridge},
      YEAR = {2008},
     PAGES = {xvii+347},
      ISBN = {978-0-521-87950-7},
   MRCLASS = {05-02 (05C82 34C15 34D06 60K99 90B10 91D30)},
  MRNUMBER = {2797803},
      NOTE = {\url{https://doi.org/10.1017/CBO9780511791383}},
}

@article{Jivkov2014Anetworkmodel,
	author = {Jivkov, Andrey P. and Xiong, Qingrong},
	date = {2014/10/01},
      NOTE = {\url{https://doi.org/10.1007/s11242-014-0360-1}},
	id = {Jivkov2014},
	isbn = {1573-1634},
        journal = {Transp. Porous Media},
	fjournal = {Transport in Porous Media},
	number = {1},
	pages = {83--104},
	title = {A Network Model for Diffusion in Media with Partially Resolvable Pore Space Characteristics},
	volume = {105},
	year = {2014},
}

@article{Raj2012Anetworkdiffusionmodel,
	author = {Raj, Ashish and Kuceyeski, Amy and Weiner, Michael},
	journal = {Neuron},
        fjournal = {Neuron},
	month = {Mar},
	number = {6},
	pages = {1204--1215},
	title = {A network diffusion model of disease progression in dementia.},
	volume = {73},
       NOTE = {\url{https://doi.org/10.1016/j.neuron.2011.12.040}},
	year = {2012}
}

@article {Leugering2023OptimalControl,
    AUTHOR = {Leugering, G\"unter and Mophou, Gis\`ele and Moutamal, Maryse
              and Warma, Mahamadi},
     TITLE = {Optimal control problems of parabolic fractional
              {S}turm-{L}iouville equations in a star graph},
   JOURNAL = {Math. Control Relat. Fields},
  FJOURNAL = {Mathematical Control and Related Fields},
    VOLUME = {13},
      YEAR = {2023},
    NUMBER = {2},
     PAGES = {771--807},
      ISSN = {2156-8472,2156-8499},
   MRCLASS = {49J20 (26A33 35J20 49K20)},
  MRNUMBER = {4534588},
MRREVIEWER = {Xiaodong\ Yan},
      NOTE = {\url{https://doi.org/10.3934/mcrf.2022015}},
}

@article {Daoud2024Aclass,
    AUTHOR = {Daoud, Maha and Laamri, El-Haj and Baalal, Azeddine},
     TITLE = {A class of fractional parabolic reaction--diffusion systems
              with control of total mass: theory and numerics},
   JOURNAL = {J. Pseudo-Differ. Oper. Appl.},
  FJOURNAL = {Journal of Pseudo-Differential Operators and Applications},
    VOLUME = {15},
      YEAR = {2024},
    NUMBER = {1},
     PAGES = {Paper No. 18, 36},
      ISSN = {1662-9981,1662-999X},
   MRCLASS = {35R11 (35B45 35K59 47H10)},
  MRNUMBER = {4709430},
      NOTE = {\url{https://doi.org/10.1007/s11868-023-00576-w}},
}

@misc{Daoud2025classp,
      title={A class of parabolic reaction-diffusion systems governed by spectral fractional Laplacians : Analysis and numerical simulations}, 
      author={Maha Daoud},
      year={2025},
      eprint={2502.13771},
      archivePrefix={arXiv},
      primaryClass={math.AP},
      NOTE = {\url{https://arxiv.org/abs/2502.13771}},
}

@incollection {Vazquez2017Themathematical,
    AUTHOR = {V\'azquez, Juan Luis},
     TITLE = {The mathematical theories of diffusion: nonlinear and
              fractional diffusion},
 BOOKTITLE = {Nonlocal and nonlinear diffusions and interactions: new
              methods and directions},
    SERIES = {Lecture Notes in Math.},
    VOLUME = {2186},
     PAGES = {205--278},
 PUBLISHER = {Springer, Cham},
      YEAR = {2017},
      ISBN = {978-3-319-61493-9; 978-3-319-61494-6},
   MRCLASS = {35K57 (35R11)},
  MRNUMBER = {3588125},
MRREVIEWER = {Jeffrey\ R.\ Anderson},
      NOTE = {\url{https://doi.org/10.1007/978-3-319-61494-6_5}},
}

@article{Riascos2014Fractional,
  title = {Fractional dynamics on networks: Emergence of anomalous diffusion and L\'evy flights},
  author = {Riascos, A. P. and Mateos, Jos\'e L.},
  journal = {Phys. Rev. E},
  fjournal = {Physical Review E},
  volume = {90},
  issue = {3},
  pages = {032809},
  numpages = {7},
  year = {2014},
  month = {Sep},
  publisher = {American Physical Society},
      NOTE = {\url{https://doi.org/10.1103/PhysRevE.90.032809}},
}

@article {Miller2006Onthecontrollability,
    AUTHOR = {Miller, Luc},
     TITLE = {On the controllability of anomalous diffusions generated by
              the fractional {L}aplacian},
   JOURNAL = {Math. Control Signals Systems},
  FJOURNAL = {Mathematics of Control, Signals, and Systems},
    VOLUME = {18},
      YEAR = {2006},
    NUMBER = {3},
     PAGES = {260--271},
      ISSN = {0932-4194,1435-568X},
   MRCLASS = {93B05 (35F10 47D06)},
  MRNUMBER = {2272076},
MRREVIEWER = {Jun-Min\ Wang},
      NOTE = {\url{https://doi.org/10.1007/s00498-006-0003-3}},
}

@article {Henry2006Anomalous,
    AUTHOR = {Henry, B. I. and Langlands, T. A. M. and Wearne, S. L.},
     TITLE = {Anomalous diffusion with linear reaction dynamics: from
              continuous time random walks to fractional reaction-diffusion
              equations},
   JOURNAL = {Phys. Rev. E (3)},
  FJOURNAL = {Physical Review E. Statistical, Nonlinear, and Soft Matter
              Physics},
    VOLUME = {74},
      YEAR = {2006},
    NUMBER = {3},
     PAGES = {031116, 15},
      ISSN = {1539-3755,1550-2376},
   MRCLASS = {82C41 (35K57)},
  MRNUMBER = {2282122},
      NOTE = {\url{https://doi.org/10.1103/PhysRevE.74.031116}},
}

@article {Somathilake2018Aspacefractional,
    AUTHOR = {Somathilake, Lekam Watte and Burrage, Kevin},
     TITLE = {A space-fractional-reaction-diffusion model for pattern
              formation in coral reefs},
   JOURNAL = {Cogent Math. Stat.},
  FJOURNAL = {Cogent Mathematics \& Statistics},
    VOLUME = {5},
      YEAR = {2018},
    NUMBER = {1},
     PAGES = {Art. ID 1426524, 21},
      ISSN = {2574-2558},
   MRCLASS = {92C15 (35R11 37N25 65M70)},
  MRNUMBER = {3782923},
      NOTE = {\url{https://doi.org/10.1080/23311835.2018.1426524}},
}

@book {Danko2017Modelelements,
    AUTHOR = {Danko, George L.},
     TITLE = {Model elements and network solutions of heat, mass and
              momentum transport processes},
    SERIES = {Heat and Mass Transfer},
 PUBLISHER = {Springer-Verlag, Berlin},
      YEAR = {2017},
     PAGES = {xvii+251},
      ISBN = {978-3-662-52929-4; 978-3-662-52931-7},
   MRCLASS = {76-02 (80-02)},
  MRNUMBER = {3560892},
      NOTE = {\url{https://doi.org/10.1007/978-3-662-52931-7}},
}

@article {Metzler2000Therandomwalksguide,
    AUTHOR = {Metzler, Ralf and Klafter, Joseph},
     TITLE = {The random walk's guide to anomalous diffusion: a fractional
              dynamics approach},
   JOURNAL = {Phys. Rep.},
  FJOURNAL = {Physics Reports. A Review Section of Physics Letters},
    VOLUME = {339},
      YEAR = {2000},
    NUMBER = {1},
     PAGES = {77},
      ISSN = {0370-1573,1873-6270},
   MRCLASS = {82C31 (82C70)},
  MRNUMBER = {1809268},
      NOTE = {\url{https://doi.org/10.1016/S0370-1573(00)00070-3}},
}

@book {Evangelista2018Fractional,
    AUTHOR = {Evangelista, Luiz Roberto and Lenzi, Ervin Kaminski},
     TITLE = {Fractional diffusion equations and anomalous diffusion},
 PUBLISHER = {Cambridge University Press, Cambridge},
      YEAR = {2018},
     PAGES = {xiii+345},
      ISBN = {978-1-107-14355-5},
   MRCLASS = {35-02 (26A33 35R11 60G22 74-02 74N25)},
  MRNUMBER = {3793188},
MRREVIEWER = {Krzysztof\ Rogowski},
      NOTE = {\url{https://doi.org/10.1017/9781316534649}},
}

@book {Evans2010Partial,
    AUTHOR = {Evans, Lawrence C.},
     TITLE = {Partial differential equations},
    SERIES = {Graduate Studies in Mathematics},
    VOLUME = {19},
   EDITION = {Second},
 PUBLISHER = {American Mathematical Society, Providence, RI},
      YEAR = {2010},
     PAGES = {xxii+749},
      ISBN = {978-0-8218-4974-3},
   MRCLASS = {35-01},
  MRNUMBER = {2597943},
MRREVIEWER = {Diego\ M.\ Maldonado},
      NOTE = {\url{https://doi.org/10.1090/gsm/019}},
}

@book {Lions1972Nonhomogeneous,
    AUTHOR = {Lions, J.-L. and Magenes, E.},
     TITLE = {Non-homogeneous boundary value problems and applications.
              {V}ol. {I}},
    SERIES = {Die Grundlehren der mathematischen Wissenschaften},
    VOLUME = {Band 181},
 PUBLISHER = {Springer-Verlag, New York-Heidelberg},
      YEAR = {1972},
     PAGES = {xvi+357},
   MRCLASS = {35JXX (35KXX 35LXX 46E35)},
  MRNUMBER = {350177},
      NOTE = {Translated from the French by P. Kenneth. \url{https://doi.org/10.1007/978-3-642-65161-8}},
}

@book {Lunardi2018Interpolation,
    AUTHOR = {Lunardi, Alessandra},
     TITLE = {Interpolation theory},
    SERIES = {Appunti. Scuola Normale Superiore di Pisa (Nuova Serie)
              [Lecture Notes. Scuola Normale Superiore di Pisa (New
              Series)]},
    VOLUME = {16},
   EDITION = {Third},
 PUBLISHER = {Edizioni della Normale, Pisa},
      YEAR = {2018},
     PAGES = {xiv+199},
      ISBN = {978-88-7642-639-1; 978-88-7642-638-4},
   MRCLASS = {46M35 (46-02 46B70 47D06 47F05)},
  MRNUMBER = {3753604},
      NOTE = {\url{https://doi.org/10.1007/978-88-7642-638-4}},
}

@article {Faheem2023Acollocation,
    AUTHOR = {Faheem, Mo and Khan, Arshad},
     TITLE = {A collocation method for time-fractional diffusion equation on
              a metric star graph with {$\eta$} edges},
   JOURNAL = {Math. Methods Appl. Sci.},
  FJOURNAL = {Mathematical Methods in the Applied Sciences},
    VOLUME = {46},
      YEAR = {2023},
    NUMBER = {8},
     PAGES = {8895--8914},
      ISSN = {0170-4214,1099-1476},
   MRCLASS = {65T60 (35A22 35R11)},
  MRNUMBER = {4589847},
MRREVIEWER = {Libo\ Feng},
      NOTE = {\url{https://doi.org/10.1002/mma.9023}},
       URL = {https://doi.org/10.1002/mma.9023},
}

@article {Kumari2025Finite,
    AUTHOR = {Kumari, Shweta and Mehra, Mani and Mehandiratta, Vaibhav},
     TITLE = {Finite difference approximation of time-fractional
              advection-diffusion equation on a metric star graph},
   JOURNAL = {Int. J. Comput. Math.},
  FJOURNAL = {International Journal of Computer Mathematics},
    VOLUME = {102},
      YEAR = {2025},
    NUMBER = {12},
     PAGES = {2032--2050},
      ISSN = {0020-7160,1029-0265},
   MRCLASS = {65M06},
  MRNUMBER = {4996704},
      NOTE = {\url{https://doi.org/10.1080/00207160.2025.2531521}},
       URL = {https://doi.org/10.1080/00207160.2025.2531521},
}

@article {Goloshchapova2021Anonlinear,
    AUTHOR = {Goloshchapova, Nataliia},
     TITLE = {A nonlinear {K}lein-{G}ordon equation on a star graph},
   JOURNAL = {Math. Nachr.},
  FJOURNAL = {Mathematische Nachrichten},
    VOLUME = {294},
      YEAR = {2021},
    NUMBER = {9},
     PAGES = {1742--1764},
      ISSN = {0025-584X,1522-2616},
   MRCLASS = {35Q40},
  MRNUMBER = {4333890},
      NOTE = {\url{https://doi.org/10.1002/mana.201900526}},
       URL = {https://doi.org/10.1002/mana.201900526},
}

@article {Dutykh2018Wave,
    AUTHOR = {Dutykh, Denys and Caputo, Jean-Guy},
     TITLE = {Wave dynamics on networks: method and application to the
              sine-{G}ordon equation},
   JOURNAL = {Appl. Numer. Math.},
  FJOURNAL = {Applied Numerical Mathematics. An IMACS Journal},
    VOLUME = {131},
      YEAR = {2018},
     PAGES = {54--71},
      ISSN = {0168-9274,1873-5460},
   MRCLASS = {65M06 (05C82 35L71 35R02)},
  MRNUMBER = {3807170},
      NOTE = {\url{https://doi.org/10.1016/j.apnum.2018.03.010}},
       URL = {https://doi.org/10.1016/j.apnum.2018.03.010},
}

@article {Goodman2025QGlAB,
    AUTHOR = {Goodman, Roy H. and Conte, Grace and Marzuola, Jeremy L.},
     TITLE = {Q{GLAB}: a {MATLAB} package for computations on quantum
              graphs},
   JOURNAL = {SIAM J. Sci. Comput.},
  FJOURNAL = {SIAM Journal on Scientific Computing},
    VOLUME = {47},
      YEAR = {2025},
    NUMBER = {2},
     PAGES = {B428--B453},
      ISSN = {1064-8275,1095-7197},
   MRCLASS = {65Y15 (35-02 35J10 35P05 35Q55 65M70 81Q35)},
  MRNUMBER = {4885029},
      NOTE = {\url{https://doi.org/10.1137/23M1627729}},
       URL = {https://doi.org/10.1137/23M1627729},
}

@article {Goodman2019NLSbifurcations,
    AUTHOR = {Goodman, Roy H.},
     TITLE = {N{LS} bifurcations on the bowtie combinatorial graph and the
              dumbbell metric graph},
   JOURNAL = {Discrete Contin. Dyn. Syst.},
  FJOURNAL = {Discrete and Continuous Dynamical Systems},
    VOLUME = {39},
      YEAR = {2019},
    NUMBER = {4},
     PAGES = {2203--2232},
      ISSN = {1078-0947,1553-5231},
   MRCLASS = {05C90 (81Q05)},
  MRNUMBER = {3927510},
      NOTE = {\url{https://doi.org/10.3934/dcds.2019093}},
       URL = {https://doi.org/10.3934/dcds.2019093},
}

@article {Kovacs2021Stochastics,
    AUTHOR = {Kov\'acs, M. and Sikolya, E.},
     TITLE = {Stochastic reaction-diffusion equations on networks},
   JOURNAL = {J. Evol. Equ.},
  FJOURNAL = {Journal of Evolution Equations},
    VOLUME = {21},
      YEAR = {2021},
    NUMBER = {4},
     PAGES = {4213--4260},
      ISSN = {1424-3199,1424-3202},
   MRCLASS = {60H15 (35R02 35R60 47D06)},
  MRNUMBER = {4350573},
      NOTE = {\url{https://doi.org/10.1007/s00028-021-00719-w}},
       URL = {https://doi.org/10.1007/s00028-021-00719-w},
}

@article {Mehandiratta2021Optimalcontrol,
    AUTHOR = {Mehandiratta, Vaibhav and Mehra, Mani and Leugering, Gunter},
     TITLE = {Optimal control problems driven by time-fractional diffusion
              equations on metric graphs: optimality system and finite
              difference approximation},
   JOURNAL = {SIAM J. Control Optim.},
  FJOURNAL = {SIAM Journal on Control and Optimization},
    VOLUME = {59},
      YEAR = {2021},
    NUMBER = {6},
     PAGES = {4216--4242},
      ISSN = {0363-0129,1095-7138},
   MRCLASS = {49K20 (26A33 35R11 49J20 49M41 65M06 93C20)},
  MRNUMBER = {4334537},
      NOTE = {\url{https://doi.org/10.1137/20M1340332}},
}

@article {Bonito2017Theapproximation,
    AUTHOR = {Bonito, Andrea and Lei, Wenyu and Pasciak, Joseph E.},
     TITLE = {The approximation of parabolic equations involving fractional
              powers of elliptic operators},
   JOURNAL = {J. Comput. Appl. Math.},
  FJOURNAL = {Journal of Computational and Applied Mathematics},
    VOLUME = {315},
      YEAR = {2017},
     PAGES = {32--48},
      ISSN = {0377-0427,1879-1778},
   MRCLASS = {65M60 (35R11 65M12)},
  MRNUMBER = {3583668},
MRREVIEWER = {JiChun\ Li},
      NOTE = {\url{https://doi.org/10.1016/j.cam.2016.10.016}},
       URL = {https://doi.org/10.1016/j.cam.2016.10.016},
}
\end{document}